\documentclass[12pt]{amsart}

\usepackage{amsmath,amssymb,latexsym}
\usepackage{fullpage}
\usepackage{etoolbox}
\usepackage{enumitem}
\usepackage{color}
\usepackage[colorlinks,linkcolor=blue,citecolor=magenta, pagebackref]{hyperref}
\usepackage{soul}
\usepackage{microtype}
\usepackage{yhmath}
\usepackage{graphicx}
\usepackage{stackrel}
\usepackage{mathtools}
\usepackage{array}
\usepackage{float}
\usepackage{complexity}
\usepackage[sharp]{easylist}
\usepackage{tikz}
\usetikzlibrary{decorations.markings}
\usetikzlibrary{calc}
\usetikzlibrary{shapes.geometric}
\usetikzlibrary{graphs}
\usepackage{subcaption}

\newtheorem{theorem}{Theorem}[section]
\newtheorem{proposition}[theorem]{Proposition}

\newtheorem{lemma}[theorem]{Lemma}

\theoremstyle{definition}
\newtheorem{remark}{Remark}

\newtheorem{example}{Example}
\newtheorem{question}{Question}

\def\R{\mathbb{R}}
\def\Z{\mathbb{Z}}

\def\F{\mathcal{F}}
\def\x{\boldsymbol{x}}

\def\HH{\mathcal{H}}
\def\K{\mathsf{K}}
\def\L{\mathsf{L}}
\def\B{\mathsf{B}}
\def\N{\mathsf{N}}

\renewcommand{\geq}{\geqslant}
\renewcommand{\leq}{\leqslant}
\renewcommand{\preceq}{\preccurlyeq}
\renewcommand{\succeq}{\succcurlyeq}
\newcommand{\uleq}{\mathbin{\rotatebox[origin=c]{45}{$\leq$}}}
\newcommand{\dleq}{\mathbin{\rotatebox[origin=c]{-45}{$\leq$}}}

\def\Hom{\operatorname{Hom}}

\def\susp{\operatorname{susp}}

\def\CN{\operatorname{CN}}
\def\KG{\operatorname{KG}}
\def\SG{\operatorname{SG}}
\def\conn{\operatorname{conn}}
\def\coind{\operatorname{coind}}
\def\ind{\operatorname{ind}}
\def\hind{\operatorname{ch-ind}_{\Z_2}}
\def\Xind{\operatorname{Xind}}
\def\cd{\operatorname{cd}}
\def\zig{\operatorname{zig}}
\def\sd{\operatorname{sd}}
\def\lb{\operatorname{lb}}
\def\param{\operatorname{param}}

\def\Bor{\operatorname{Bor}}
\def\height{\operatorname{height}}

\newcounter{relctr} 
\everydisplay\expandafter{\the\everydisplay\setcounter{relctr}{0}} 
\AtBeginDocument{\let\originallabel\label} 

\newcounter{foo}

\newcommand\labelrel[2]{%
  \begingroup
    \refstepcounter{relctr}%
    \stackrel{\textnormal{(\roman{relctr})}}{\mathstrut{#1}}%
    \originallabel{#2}%
  \endgroup
}

\newcommand\labeldiag[1]{
\begin{minipage}{1cm}
        \begin{equation}
        \label{#1}
        \end{equation}
\end{minipage}
 }

\makeatletter
    \newcommand{\customlabel}[2]{%
        \protected@write \@auxout {}{%
            \string \newlabel {#1}{{#2}{\thepage}{#2}{#1}{}}%
        }%
        \hypertarget{#1}{#2}%
    }
\makeatother

\title{Box complexes: at the crossroad of \\ graph theory and topology}

\author{Hamid Reza Daneshpajouh}
\address{HR. Daneshpajouh,
School of Mathematical Sciences, University of Nottingham Ningbo China, 199 Taikang
East Road, Ningbo 315100, China}
\email{Hamid-Reza.Daneshpajouh@nottingham.edu.cn}

\author{Fr\'ed\'eric Meunier}
\address{F. Meunier, CERMICS, \'Ecole des Ponts, 77455 Marne-la-Vall\'ee CEDEX, France}
\email{frederic.meunier@enpc.fr}

\begin{document}

\begin{abstract}
Various simplicial complexes can be associated with a graph. Box complexes form an important families of such simplicial complexes and are especially useful for providing lower bounds on the chromatic number of the graph via some of their topological properties. They provide thus a fascinating topic mixing topology and discrete mathematics. This paper is intended to provide an up-do-date survey on box complexes. It is based on classical results and recent findings from the literature, but also establishes new results improving our current understanding of the topic, and identifies several challenging open questions.
\keywords{box complexes, coloring, graph theory, topology}
\end{abstract}

\subjclass[2020]{05C15,	55P10, 68Q17}

\maketitle

\section{Introduction}

Since the 1978 breakthrough paper by Lov\'asz solving the Kneser conjecture~\cite{lovasz1978kneser}, various simplicial complexes associated with graphs have been studied, in relations with other combinatorial problems or in their own right. The search for good topological bounds on the chromatic number of graphs has been a great stimulation in this area, and has been at the origin of an especially prominent family of simplicial complexes, namely that of box complexes. (Simplicial complexes will play a ubiquitous role in this paper; see Matou{\v{s}}ek~\cite[Chapter 1]{matousek2008using} for a gentle introduction to these mathematical objects.) A {\em box complex} associated with a graph is a simplicial complex whose simplices are its (not necessarily induced) complete bipartite subgraphs. This is just a rough definition, especially because we do not explain the status of an empty part, and this gives actually freedom for considering various types of box complexes. The group $\Z_2$ acts freely on box complexes by exchanging the two parts, which gives them interesting features, and which allows the use of elementary results from equivariant topology.

The popularity of box complexes can be explained by the simplicity of their definition, but also for other reasons: they provide among the best topological lower bounds on the chromatic number, their relation with other simplicial complexes is well understood, and they form an intriguing connection between discrete mathematics and topology

The objective of this paper is to form an up-to-date survey of the properties of box-complexes. The diagram of Figure~\ref{fig} is the main object around which our work is organized. It shows how various bounds on the chromatic number, especially topological lower bounds in relation with box complexes, compare: each arrow represents a $\leq$ inequality, from the smaller parameter to the larger. The meaning of the various expressions in the diagram is given later in the paper. However, we already emphasize that we focus only on two kinds of box complexes: $\B(G)$ and $\B_0(G)$. Other box complexes have been considered in the literature but all of them are $\Z_2$-homotopy equivalent to one or the other.

\setcounter{foo}{\theequation}
\setcounter{equation}{0}
\renewcommand*\theequation{\alph{equation}}

\begin{figure}
\begin{tikzpicture}
        \node at (0,-11) (X1) {{$\omega(G)$}};
        \node at (4,-11) (X2) {{$\conn(\B(G))+3$}};
        \node at (2,-9.5) (X3) {{$\coind(\B(G))+2$}};
        \node at (6,-9.5) (X4) {{$\conn(\B_0(G))+2$}};
        \node at (0,-8) (Y1) {{$\cd(\HH)$}};
        \node at (4,-8) (Y2) {{$\coind(\B_0(G))+1$}};
        \node at (10,-8) (Y3) {{$\conn_{\Z_2}(\B(G))+3=\conn_{\Z_2}(\B_0(G))+2$}};
        \node at (6,-6.5) (Z1) {{$\hind(\B(G))+2=\hind(\B_0(G))+1$}};
        \node at (6,-5) (Z2) {{$\ind(\B_0(G))+1$}};
        \node at (2.6,-3.5) (Z3) {{$b(G)$}};
        \node at (6,-3.5) (Z4) {{$\ind(\B(G))+2$}};
        \node at (6,-2) (Z5) {{$\Xind(\Hom(K_2,G))+2$}};
        \node at (6,-0.5) (W1) {{$\zig(G)$}};
        \node at (6,1) (W2) {{$\chi(G)$}};
        \node at (11.5,-7.2) (T) {\textcolor{purple}{\small Borsuk--Ulam boundary}};
       
        \draw[->, thick] (X1) -- (X3);
        \node at (0.85,-9.9) (ar1) {\labeldiag{ar1}};
        \draw[->, thick] (X2) -- (X3);
        \node at (3.6,-9.9) (ar2) {\labeldiag{ar2}};
        \draw[->, thick] (X2) -- (X4);
        \node at (4.85,-9.9) (ar3) {\labeldiag{ar3}};
        \draw[->, thick] (X3) -- (Y2); 
        \node at (2.85,-8.4) (ar4) {\labeldiag{ar4}};
        \draw[ ->, thick] (X4) -- (Y2); 
        \node at (5.6,-8.4) (ar5) {\labeldiag{ar5}};
        \draw[ ->, thick] (X4) -- (Y3);
        \node at (7.85,-8.4) (ar6) {\labeldiag{ar6}};
        \draw[ ->, thick] (Y1) -- (Y2);
        \node at (1.7,-7.55) (ar7) {\labeldiag{ar7}};
        \draw[ ->, thick] (Y2) -- (Z1);
        \node at (4.85,-6.9) (ar8) {\labeldiag{ar8}};
        \draw[->, thick] (Y3) -- (Z1); 
        \stepcounter{equation}
        \node at (8.8,-6.9) (a9) {\labeldiag{ar9}};
       
        \draw [purple, dashed, very thick] (-0.5,-7.35) -- (13.5,-7.35);
       
        \draw[ ->, thick] (Z1) -- (Z2); 
        \node at (5.9,-5.6) (a10) {\labeldiag{ar10}};
        \draw[ ->, thick] (Z2) -- (Z4); 
        \node at (5.9,-4.1) (a11) {\labeldiag{ar11}};
        \draw[ ->, thick] (Z4) -- (Z3); 
        \node at (4.2,-3) (a12) {\labeldiag{ar12}};
        \draw[ ->, thick] (Z4) -- (Z5); 
        \node at (5.8,-2.65) (a13) {\labeldiag{ar13}};
        \draw[ ->, thick] (Z5) -- (W1);
        \node at (5.9,-1.1) (a14) {\labeldiag{ar14}};
        \draw[ ->, thick] (W1) -- (W2); 
        \node at (5.9,0.4) (a15) {\labeldiag{ar15}};
\end{tikzpicture}
\caption{\label{fig}Each arrow represents a $\leq$ inequality, from the smaller parameter to the larger. Apart from the clique number, the parameters below the ``Borsuk--Ulam boundary'' requires the Borsuk--Ulam theorem---or a ``stronger'' statement---to be established as a lower bound on the chromatic number. This is further discussed in Section~\ref{sec:comput}.}
\end{figure}
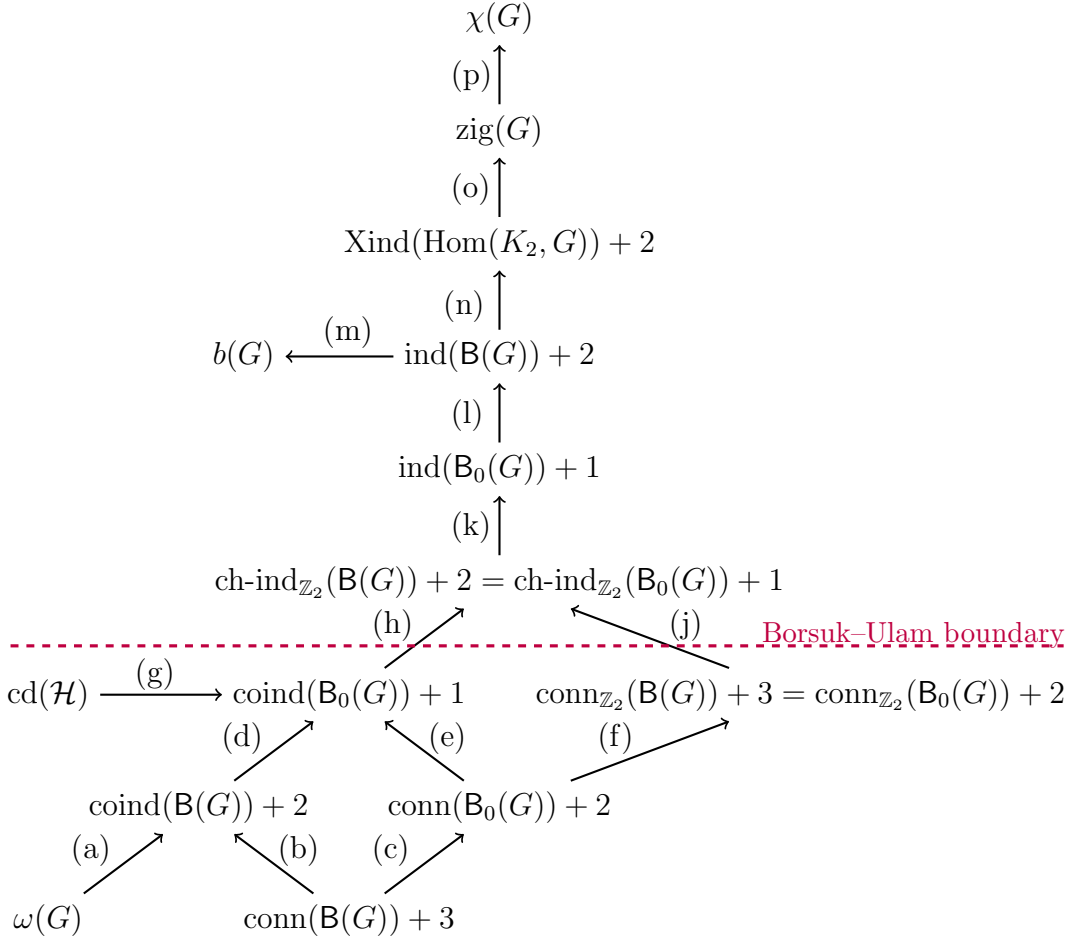

All arrows (and all non-arrows) will thoroughly be discussed. To achieve this, we not only collect results from the literature but also prove new results, such as possible gaps between consecutive bounds in the diagram. For instance, Simony, Tardos, and Vrecica~\cite{simonyi2009local} have shown that the gap in inequality~\eqref{ar4} can be equal to $1$, but they left open the question whether the gap can be larger. We contribute to that question by constructing infinitely many graphs for which the gap is $2$; see Section~\ref{subsec:arrows}, where inequality~\eqref{ar4} is discussed, and Section~\ref{subsec:Briesk} for the topological result used for the construction.

Apart from the improved understanding of the arrows and non-arrows of Figure~\ref{fig}, the paper provides other contributions, like: the decidability of $\conn(\B_0(G))$ (Theorem~\ref{thm:dec_B0}); a complexity result showing the equivalence between the Borsuk--Ulam theorem and the inequality $\coind(\B(G))+2 \leq \chi(G)$ (Theorem~\ref{thm:complex-coind}); a theorem relating the box complex of the join of two graphs to the join of their box complexes (Theorem~\ref{thm:join-B}).

The paper is organized as follows. 

\smallskip

\begin{easylist}\ListProperties(Style1*=\scshape$\star$\hspace{0.1cm}, Style2*=$-$, Hide=2, FinalSpace=2em, Indent=0.5cm, Space1=0.2cm, Space1*=0.2cm, Space2=0.2cm, Space2*=0.2cm)

# Section~\ref{sec:first} deals with the definition of the box complex $\B(G)$ and presents its main properties and its relevance. It is a short section aiming at being a gentle and self-contained introduction to the topic.

# Section~\ref{sec:complex} introduces the other box complex considered in the paper, namely $\B_0(G)$, as well as two other complexes that can be built from graphs, namely the Hom complex and the neighborhood complex. A few topological properties of the box complex $\B(G)$ when $G$ is a Kneser graph, a Schrijver graph, or a chordal graph are given. The first two kinds of graphs are classical ones in the area of topological combinatorics.

# Section~\ref{sec:parameter} provides the definition of various parameters of graphs and complexes. All parameters that appear in Figure~\ref{fig} are in particular defined in that section. The section is divided into three subsections: the first on combinatorial parameters, the second on topological parameters, the third on the comparisons between the parameters attached to simplicial complexes.

# Section~\ref{sec:comput} discusses the computability of the various parameters and their relation with the Borsuk--Ulam theorem. As explained in that section, these two topics are not independent.

# Section~\ref{sec:product} is a short section on (categorical) products of graphs and topological spaces. Even if this section is short, it is an important topic dealing mostly with Hedetniemi's conjecture, a central conjecture in graph theory (now disproved).

# Section~\ref{sec:join} is about joins of graphs and topological spaces. The join operation can also be seen as a kind of product, but which is more flexible and has more applications than the product studied in the previous section.

# Section~\ref{sec:comment} provides a thorough explanation of the diagram of Figure~\ref{fig}. For each arrow, the current knowledge on how large the gap can be is discussed. The non-arrow are also discussed (is an arrow absent because there is no way to order consistently the parameters, or by lack of knowledge?).

# Section~\ref{sec:compl} gathers complementary remarks, and collects the main open questions met in the survey.
\end{easylist}

\medskip

Throughout the paper, all graphs are simple, i.e., they have no parallel edges and no loops, and finite. 

\section*{Acknowledgments}
The authors are grateful to Moishe Kohan~\cite{moishe} for pointing out the existence of the Brieskorn manifolds. They thank Elba Garcia-Failde and Bram Petri for clarifying email messages that helped the writing of the proof of Lemma~\ref{lem:Briesk}, Roman Karasev for pointing out the reference to a theorem by Conner and Floyd in the comment of inequality~\eqref{ar10}, and Anton Dochtermann for clarifications regarding Hom complexes of directed graphs. They also thank Ishay Haviv and G\'abor Simonyi for many helpful comments on the first version of the paper, especially for providing complementary references, correcting a few technical inaccuracies, and spotting some typos.

 Part of this work was done when the first author was at the School of Mathematics of IPM as a guest researcher in Spring 2021 and received support for his research by this institution.

\setcounter{equation}{\thefoo}
\renewcommand*\theequation{\arabic{equation}}

\section{Box complexes: definition, relevance, and main properties}\label{sec:first}

\subsection{Definition} Let $G$ be a graph. For a subset $A\subseteq V(G)$, let
\[
\CN_G(A) = \{v\in V(G) \colon av\in E(G)\; \text{for all $a\in A$}\} \, .
\]
It is the set of \emph{common neighbors} of $A$. When there is no ambiguity, we will write $\CN$ instead of $\CN_G$. Note that $\CN(\varnothing)=V(G)$ and that since $G$ has no loops, we have $\CN(A)\subseteq V\setminus A$. The \emph{box complex} of $G$, denoted by $\B(G)$, is the simplicial complex defined as follows:
\[
\B(G) = \big\{ A' \uplus A'' \colon A',A'' \subseteq V, A' \cap A'' = \varnothing, \; G[A',A''] \text{ is complete}, \;  \CN(A'),\CN(A'') \neq \varnothing \big\} \, .
\]
Its vertex set is the ``signed'' version of $V(G)$. Each vertex $v$ of $G$ becomes two vertices: $+v$ and $-v$. The notation $A' \uplus A''$ means $\{+v\colon v\in A'\} \cup \{-v\colon v\in A''\}$.
The notation $G[A',A'']$ stands for the bipartite graph with parts $A'$ and $A''$ and whose edges are all edges of $G$ with one endpoint in $A'$ and the other in $A''$. Roughly speaking, $\B(G)$ is the simplicial complex formed by all {\em bicliques}, i.e., complete bipartite subgraphs, of $G$. We also count a subset of vertices with at least one common neighbor as a complete bipartite subgraph with no edges (case when $A'$ or $A''$ is empty).

\begin{example}\label{ex:box-complete}
The box complex $\B(K_n)$ of the complete graph with $n$ vertices is the boundary of the $n$-dimensional cross-polytope to which we have removed two opposite facets. This implies in particular that $\B(K_n)$ is homotopy equivalent to the $(n-2)$-dimensional sphere $S^{n-2}$.
\end{example}

\begin{example}\label{ex:box-C2n}
    The box complex $\B(C_4)$ is formed by the disjoint union of two copies of the $3$-dimensional simplex. This implies in particular that $\B(C_4)$ is homotopy equivalent to $S^0$. This is also a special case of complete bipartite graphs dealt with in Example~\ref{ex:box-bipcomplete}.

    For $n\geq 3$, the box complex $\B(C_{2n})$ of the $2n$-cycle is homeomorphic to the disjoint union of two copies of $S^1 \times [0,1]$ (the boundary of each copy being formed by two $n$-cycles).  This implies in particular that $\B(C_{2n})$ is homotopy equivalent to two disjoint copies of $S^1$. See Figure~\ref{fig:BC6} for an illustration of the $6$-cycle and its box complex $\B(C_6)$.
\end{example}

\begin{example}\label{ex:box-C2n+1}
    The box complex $\B(C_{2n+1})$ of the $(2n+1)$-cycle is homeomorphic to $S^1 \times [0,1]$ (whose boundary is formed by two $(2n+1)$-cycles).  This implies in particular that $\B(C_{2n+1})$ is homotopy equivalent to $S^1$.
\end{example}

\begin{figure}
    \centering
    \subcaptionbox{The cycle $C_6$}[0.3\textwidth]{
\begin{tikzpicture}[scale=0.8]
	\node[circle, regular polygon,
	draw,
	regular polygon sides = 6, minimum size = 3cm] (p) at (0,0) {};
 \foreach \x in {1,2,...,6}
  \fill (p.corner \x) circle[radius=2pt];
 \node [left] (5) at (p.west) {\footnotesize $5$};
 \node [right] (2) at (p.east) {\footnotesize $2$};
 \node [above left] (6) at (p.120) {\footnotesize $6$};
 \node [above right] (1) at (p.60) {\footnotesize $1$};
 \node [below left](4) at (p.240) {\footnotesize $4$};
 \node [below right] (3) at (p.300) {\footnotesize $3$};
\end{tikzpicture} 
}
\subcaptionbox{The box complex $\B(C_6)$}[0.6\textwidth]{
\begin{tikzpicture}[scale=1.3]
\node (A) at (1,0,0) {};
\node (B) at (-1,0,0) {};
\node (C) at (0.3,0,1.3) {};
\coordinate (H) at (0,2,0);

\draw($(B)$) -- ($(A)$) -- ($(C)$) -- cycle ;
\draw($(B)+(H)$) -- ($(A)+(H)$) -- ($(C)+(H)$) -- cycle ;
\draw[fill=lightgray,fill opacity=.7] ($(B)$) -- ($(A)+ (H)$) -- ($(B)+ (H)$) -- cycle ;
\draw[fill=lightgray,fill opacity=.7] ($(B)$) -- ($(A)+ (H)$) -- ($(A)$) -- cycle ;
\draw[fill=lightgray,fill opacity=.7] ($(A)$) -- ($(C)+ (H)$) -- ($(A)+ (H)$) -- cycle ;
\draw[fill=lightgray,fill opacity=.7] ($(A)$) -- ($(C)$) -- ($(C)+ (H)$) -- cycle ;
\draw[fill=lightgray,fill opacity=.7] ($(C)$) -- ($(B)+ (H)$) -- ($(C)+ (H)$) -- cycle ;
\draw[fill=lightgray,fill opacity=.7] ($(C)$) -- ($(B)$) -- ($(B)+ (H)$) -- cycle ;

\node [minimum size=0.2cm]  (lA) at (1.3,0,0) {\footnotesize $-1$};
\node (lB) at (-1.3,0,0) {\footnotesize $-5$};
\node (lC) at (0.37,-0.10,1.6) {\footnotesize $-3$};
\node (lAH) at (1.3,2,0) {\footnotesize $+6$};
\node (lBH) at (-1.3,2,0) {\footnotesize $+4$};
\node (lCH) at (1.35,3.23,3.8) {\footnotesize $+2$};

\fill (A) circle[radius=1.2pt];
\fill (B) circle[radius=1.2pt];
\fill (C) circle[radius=1.2pt];
\fill ($(A)+ (H)$) circle[radius=1.2pt];
\fill ($(B)+ (H)$) circle[radius=1.2pt];
\fill ($(C)+ (H)$) circle[radius=1.2pt];
\end{tikzpicture}
\centering
\begin{tikzpicture}[scale=1.3]
\node (A) at (1,0,0) {};
\node (B) at (-1,0,0) {};
\node (C) at (0.3,0,1.3) {};
\coordinate (H) at (0,2,0);

\draw($(B)$) -- ($(A)$) -- ($(C)$) -- cycle ;
\draw($(B)+(H)$) -- ($(A)+(H)$) -- ($(C)+(H)$) -- cycle ;
\draw[fill=lightgray,fill opacity=.7] ($(B)$) -- ($(A)+ (H)$) -- ($(B)+ (H)$) -- cycle ;
\draw[fill=lightgray,fill opacity=.7] ($(B)$) -- ($(A)+ (H)$) -- ($(A)$) -- cycle ;
\draw[fill=lightgray,fill opacity=.7] ($(A)$) -- ($(C)+ (H)$) -- ($(A)+ (H)$) -- cycle ;
\draw[fill=lightgray,fill opacity=.7] ($(A)$) -- ($(C)$) -- ($(C)+ (H)$) -- cycle ;
\draw[fill=lightgray,fill opacity=.7] ($(C)$) -- ($(B)+ (H)$) -- ($(C)+ (H)$) -- cycle ;
\draw[fill=lightgray,fill opacity=.7] ($(C)$) -- ($(B)$) -- ($(B)+ (H)$) -- cycle ;

\node [minimum size=0.2cm]  (lA) at (1.3,0,0) {\footnotesize $-2$};
\node (lB) at (-1.3,0,0) {\footnotesize $-6$};
\node (lC) at (0.37,-0.1,1.6) {\footnotesize $-4$};
\node (lAH) at (1.3,2,0) {\footnotesize $+1$};
\node (lBH) at (-1.3,2,0) {\footnotesize $+5$};
\node (lCH) at (1.35,3.23,3.8) {\footnotesize $+3$};

\fill (A) circle[radius=1.2pt];
\fill (B) circle[radius=1.2pt];
\fill (C) circle[radius=1.2pt];
\fill ($(A)+ (H)$) circle[radius=1.2pt];
\fill ($(B)+ (H)$) circle[radius=1.2pt];
\fill ($(C)+ (H)$) circle[radius=1.2pt];
5\end{tikzpicture}}

\caption{ \label{fig:BC6} The $6$-cycle and its box complex}
\end{figure}
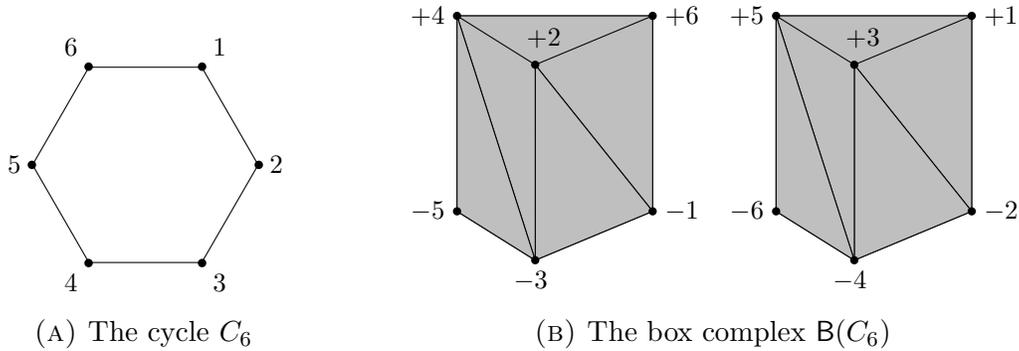

\begin{example}\label{ex:box-bipcomplete}
    The box complex $\B(K_{m,n})$ of the complete bipartite graph with parts of size respectively $m$ and $n$ is formed by two disjoint copies of the $(m+n-1)$-dimensional simplex. This implies in particular that $\B(K_{m,n})$ is homotopy equivalent to $S^0$.
\end{example}

Box complexes have been introduced by Matou\v sek and Ziegler~\cite{matousek2002topological} (motivated by former complexes introduced by Alon, Frankl, and Lov\'asz~\cite{alon1986chromatic} and by K\v r\'i\v z~\cite{kriz1992equivariant}). In Section~\ref{sec:complex}, we will see the other box complex of $G$ considered in the paper, denoted by $\B_0(G)$.

\subsection{Chromatic number and connectivity} The chromatic number is a central notion in graph theory, and likewise for connectivity in topology. It turns out that these two fundamental notions, from two seemingly distant areas of mathematics, are present in a fundamental inequality in the theory of box complexes. Before stating this inequality, we remind the definition of these two fundamental notions.

Coloring graphs is of great importance in discrete mathematics, and beyond (computer science, operations research, etc.). A map $c\colon V(G)\rightarrow [t]$ is usually called a {\em coloring}, the integers in $[t]$ being considered as {\em colors}. When $c(u)\neq c(v)$ for every pair of \emph{adjacent} vertices $u$ and $v$ (the vertices $u$ and $v$ are \emph{adjacent} if $uv\in E(G)$), the coloring is {\em proper}. The minimal $t$ for which there exists a proper coloring is the {\em chromatic number} of $G$ and is denoted by $\chi(G)$. It is a graph parameter that is not easy to compute (it is \NP-hard) and, considering its importance, any bound is welcome.

The {\em connectivity} of a topological space $X$ is a central notion in topology. It is denoted by $\conn(X)$ and is the maximal integer $d$ such that any continuous map $f$ from the $k$-dimensional sphere $S^k$ to $X$ with $k\in\{-1,0,1,\ldots,d\}$ can be extended to a continuous map $\bar f$ from the $(k+1)$-dimensional ball $B^{k+1}$ to $X$. When $\conn(X) \geq d$, we say that $X$ is {\em $d$-connected}. In this context, $S^{-1}$ is interpreted as $\varnothing$ and $B^0$ as a single point. Therefore, $(-1)$-connected means non-empty.

One of the most important facts about the box complex $\B(G)$ is that its connectivity is related to the chromatic number of $G$, making a surprising connection between topology and graph theory:
\begin{equation}\label{eq:bg}
\conn(\B(G))+3 \leq \chi(G)\,.
\end{equation}
The connectivity of a simplicial complex is not easy to compute, but inequality~\eqref{eq:bg} offers a powerful approach to the chromatic number in many situations. There are cases where the connectivity can yet be precisely computed, others where lower bounds are known, etc. An inequality similar to inequality~\eqref{eq:bg} has in particular been used by Lov\'asz to settle Kneser's conjecture. 
Our discussion about the diagram of Figure~\ref{fig} in Section~\ref{sec:comment} will actually provide a proof of inequality~\eqref{eq:bg}, but let us sketch the general idea of the proof. The box complex is a free simplicial $\Z_2$-complex. A \emph{simplicial $\Z_2$-complex} is a simplicial complex on which $\Z_2$ acts. It is \emph{free} if each orbit of its polyhedron (underlying space) is of size two. In the case of the box complex, the action is simply the exchange $A' \uplus A'' \rightarrow A'' \uplus A'$. This makes it amenable to ``equivariant'' topology and especially to tools like the Borsuk--Ulam theorem. This latter theorem states that there is no continuous map $S^d \rightarrow S^{d-1}$ that commutes with the central symmetry. 
When the graph is lifted to its box complex, the coloring is lifted to a continuous map, and it is the kind of obstruction provided by the Borsuk--Ulam theorem that prevents the graph of being colored with too few colors.

\subsection{``Universality'' of the box complex}\label{subsec:univers}

Seeing box complexes as free simplicial $\Z_2$-complexes is very useful. Things go also the other way around, as shown by the following theorem by Csorba (see also the paper by \v Zivaljevi\' c~\cite{vzivaljevic2005wi}). ($\Z_2$-homotopy equivalence, defined hereafter, implies in particular homotopy equivalence.)

\begin{theorem}[Csorba~\cite{csorba2007homotopy}]\label{thm:csorba}
Every free simplicial $\Z_2$-complex is $\Z_2$-homotopy equivalent to the box complex $\B(G)$ of some graph $G$.
\end{theorem}

A \emph{$\Z_2$-map} between two \emph{$\Z_2$-spaces} (topological spaces endowed with a $\Z_2$-action) is a map that commutes with the action. Two continuous $\Z_2$-maps $f$ and $g$ between two $\Z_2$-spaces $X$ and $Y$ are \emph{$\Z_2$-homotopic} if there exists a continuous map $h\colon X\times[0,1] \rightarrow Y$ with $h(\cdot,0)=f(\cdot)$ and $h(\cdot,1)=g(\cdot)$, and such that $h(\cdot,t)$ is a $\Z_2$-map for all $t\in[0,1]$. Note that this definition is the traditional definition of homotopic maps, except that continuous maps are replaced by continuous $\Z_2$-maps. The definition of the $\Z_2$-homotopy equivalence is then the same as the definition of homotopy equivalence with $\Z_2$-homotopy in place of homotopy.

Csorba's proof of Theorem~\ref{thm:csorba} provides an explicit and easy construction of the graph $G$: given a simplicial $\Z_2$-complex $\K$, the vertices of $G$ are the non-empty simplices of $\K$ and $\sigma\tau$ forms an edge if $\sigma$ is a face of the image of $\tau$ by the action (or vice versa). 

\begin{figure}
    \centering
    \subcaptionbox{The square seen as a $1$-dimensional simplicial complex}[0.4\textwidth]{
\begin{tikzpicture}[scale=2]
    \node[circle,fill,inner sep=1pt](a) at (0,0) {}  ;
    \node[circle,fill,inner sep=1pt](b) at (1,0) {}  ;
    \node[circle,fill,inner sep=1pt](c) at (1,1) {}  ;
    \node[circle,fill,inner sep=1pt](d) at (0,1) {}  ;
    \draw (a)--(b)--(c)--(d)--(a);
\end{tikzpicture}}
 \subcaptionbox{Its barycentric subdivision}[0.4\textwidth]{
\begin{tikzpicture}[scale=2]
    \node[circle,fill,inner sep=1pt](a) at (0,0) {}  ;
    \node[circle,fill,inner sep=1pt](ab) at (0.5,0) {}  ;
    \node[circle,fill,inner sep=1pt](b) at (1,0) {}  ;
    \node[circle,fill,inner sep=1pt](bc) at (1,0.5) {}  ;
    \node[circle,fill,inner sep=1pt](c) at (1,1) {}  ;
    \node[circle,fill,inner sep=1pt](cd) at (0.5,1) {}  ;
    \node[circle,fill,inner sep=1pt](d) at (0,1) {}  ;
    \node[circle,fill,inner sep=1pt](da) at (0,0.5) {}  ;
    \draw (a)--(b)--(c)--(d)--(a);
\end{tikzpicture}}
\subcaptionbox{The graph $G$ obtained from Csorba's construction}[0.4\textwidth]{
\begin{tikzpicture}[scale=2]
    \node[circle,fill,inner sep=1pt](a) at (0,0) {}  ;
    \node[circle,fill,inner sep=1pt](ab) at (0.5,0) {}  ;
    \node[circle,fill,inner sep=1pt](b) at (1,0) {}  ;
    \node[circle,fill,inner sep=1pt](bc) at (1,0.5) {}  ;
    \node[circle,fill,inner sep=1pt](c) at (1,1) {}  ;
    \node[circle,fill,inner sep=1pt](cd) at (0.5,1) {}  ;
    \node[circle,fill,inner sep=1pt](d) at (0,1) {}  ;
    \node[circle,fill,inner sep=1pt](da) at (0,0.5) {}  ;
    \draw (ab)--(c);
    \draw (ab)--(d);
    \draw (ab)--(cd);
    \draw (bc)--(a);
    \draw (bc)--(d);
    \draw (bc)--(da);
    \draw (cd)--(a);
    \draw (cd)--(b);
    \draw (cd)--(ab);
    \draw (da)--(c);
    \draw (da)--(b);
    \draw (da)--(bc);
\end{tikzpicture}}
\caption{Illustration of Csorba's construction underlying the proof of Theorem~\ref{thm:csorba}}
\end{figure}
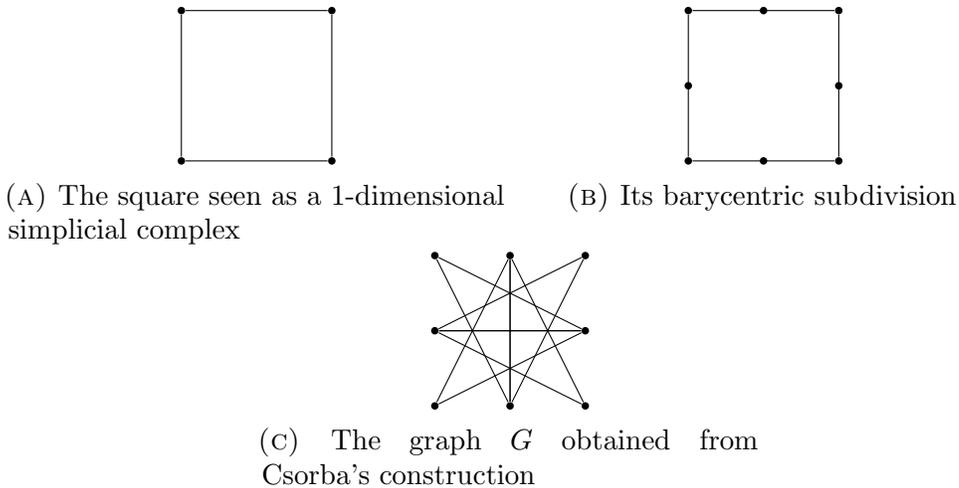

This construction suggests a few questions, which, to the authors' knowledge, have not been addressed yet. For instance, the number of vertices of $G$ in this construction is exponential in the dimension of $\K$. Would it be possible to come up with a construction requiring much less vertices in general? Or with a construction providing a graph with small chromatic number? 
\label{page:csorba}

Concrete uses of this construction seem to be scarce, and the current interest of Theorem~\ref{thm:csorba} is the existence of this graph, independently of the way it is constructed. The main message of this theorem is that there is no hope to achieve a meaningful characterization of box complexes of graphs, since these latter are actually almost as general as all free simplicial $\Z_2$-complexes. (Yet, as noted by Csorba~\cite{csorba2007homotopy}, a version of the theorem with $\Z_2$-homeomorphism does not hold.) Theorem~\ref{thm:csorba} will be used in Section~\ref{sec:comment} to prove that the gap of some inequalities between topological lower bounds can be arbitrarily large.

In the literature, there is also another ``negative'' message, arguably in the same spirit: there is no hope for characterizing the chromatic number via topological property of the box complex (in slightly different terms, this was a question by Lov\'asz in his 1978 paper solving Kneser's conjecture). Indeed, Matsushita~\cite{matsushita2017homotopy} proved that there is no homotopy invariant of the box complex that gives an upper bound for the chromatic number of a graph.

Generalizations of Csorba's construction have been proposed by Dochtermann~\cite{dochtermann2009universality} and by Dochtermann and Schultz~\cite{dochtermann2012topology}.

\section{Further results on box complexes and other relevant complexes} \label{sec:complex}

\subsection{Another box complex}\label{subsec:another} There is another box complex, denoted by $\B_0(G)$, which is probably as popular as $\B(G)$. Its definition goes as follows:
\[
\B_0(G) = \big\{ A' \uplus A'' \colon A',A'' \subseteq V(G), A' \cap A'' = \varnothing, \; G[A',A''] \text{ is complete} \big\} \, .
\]
It contains $\B(G)$. The extra simplices are the bipartite subgraphs with an empty part and whose vertices in the other part do not have a common neighbor in $G$.

\begin{example}\label{ex:box0-complete}
The box complex $\B_0(K_n)$ of the complete graph with $n$ vertices is the boundary of the $n$-dimensional cross-polytope. This implies in particular that $\B_0(K_n)$ is $\Z_2$-homeomorphic to $S^{n-1}$.
\end{example}

\begin{example}\label{ex:box0-C2n}
    The box complex $\B_0(C_4)$ of the $4$-cycle is formed by two pairs of disjoint tetrahedra, somehow arranged into a ``circle,'' where two adjacent tetrahedra share $2$ vertices. This implies in particular that $\B_0(C_4)$ is $\Z_2$-homotopy equivalent to $S^1$. See Figure~\ref{fig:my_label}. This is also a special case of complete bipartite graphs dealt with in Example~\ref{ex:box0-bipcomplete}.

    For $n \geq 3$, the box complex $\B_0(C_{2n})$ of the $2n$-cycle can be described as follows: the box complex $\B(C_{2n})$ is formed by two copies of $S^1 \times [0,1]$ (see Example~\ref{ex:box-C2n}); in $\B_0(C_{2n})$, there is an extra $(2n-1)$-dimensional simplex attached  to one boundary cycle of each copy, and there is a second extra $(2n-1)$-dimensional simplex attached to the two other boundary cycles of the copies. This implies in particular that $\B_0(C_{2n})$ is homotopy equivalent to the wedge of $S^1$ and two copies of $S^2$ (i.e., the topological space obtained by joining them at a single point).
\end{example}

\begin{example}\label{ex:box0-C2n+1}
    The box complex $\B_0(C_{2n+1})$ of the $(2n+1)$-cycle can be described as the box complex $\B(C_{2n+1})$ (see Example~\ref{ex:box-C2n+1}) with two $2n$-dimensional simplices attached to the two $(2n+1)$-cycles. This implies in particular that $\B_0(C_{2n+1})$ is $\Z_2$-homotopy equivalent to $S^2$.
\end{example}

\begin{example}\label{ex:box0-bipcomplete}
    The box complex $\B_0(K_{m,n})$ of the complete bipartite graph with parts of size respectively $m$ and $n$ is formed by two pairs of disjoint $(m+n-1)$-dimensional simplices, somehow arranged into a ``circle,'' where two adjacent $(m+n-1)$-dimensional simplices share $m$ or $n$ vertices. This implies in particular that $\B_0(K_{m,n})$ is $\Z_2$-homotopy equivalent to $S^1$.
\end{example}

Several other definitions of box complexes have been proposed in the literature. In the paper by Matou\v sek and Ziegler~\cite{matousek2002topological}, many other definitions are considered, but by results of Csorba and  Zivaljevi\'c, they are all $\Z_2$-homotopy equivalent to $\B(G)$ or to $\B_0(G)$ (see, e.g.,~\cite{vzivaljevic2005wi}). This explains why $\B(G)$ and $\B_0(G)$ have attracted most of the attention devoted to box complexes in the literature, and motivated the focus on them.

The next theorem establishes the link between the two box complexes considered in this paper. Here, the notation $\susp(\cdot)$ stands for the standard ``suspension'' operation from topology, which is defined as follows. Let $X$ be a topological space. The suspension of $X$, denoted by $\susp(X)$, is the quotient space $(X \times [-1,1])/(X\times \{-1\}, X \times \{1\})$, which corresponds to shrinking in $X\times[-1,1]$ all points $(x,t)$ with $t=-1$ to a single point, and same thing for all points with $t=1$. Suspension, and the more general ``join operation,'' will be further discussed in Section~\ref{sec:join}.

\begin{theorem}[Csorba~\cite{csorba2007homotopy}]\label{thm:susp}
For every graph $G$, the complexes $\susp(\B(G))$ and $\B_0(G)$ are $\Z_2$-homotopy equivalent.
\end{theorem}


\subsection{Other complexes}

Apart from box complexes, the {\em Hom complex}, denoted by $\Hom(K_2,G)$, and the {\em neighborhood complex}, denoted by $\N(G)$, have also played a prominent role in topological combinatorics. The Hom complex has been introduced (in a slightly different setting, and for the more general case of hypergraphs) by Alon, Frankl, and Lov\'asz~\cite{alon1986chromatic}. The neighborhood complex has been introduced by Lov\'asz  in his 1978 foundational paper. The Hom complex $\Hom(K_2,G)$ is the partial ordered set (poset) defined as
\[
\Hom(K_2,G) =  \big\{ A' \uplus A'' \colon A',A'' \subseteq V, A' \cap A'' = \varnothing, \; G[A',A''] \text{ is complete}, \;  A', A'' \neq \varnothing \big\}
\]
equipped with the following partial order: $(A',A'') \preceq (B',B'')$ if $A' \subseteq B'$ and $A'' \subseteq B''$. Note that in contrast with the previous definitions, we do not deal here with a simplicial complex but with a poset. In the literature, several options have been taken regarding Hom complexes (CW-complexes, simplicial complexes). We follow the option chosen by Dochtermann~\cite{Dochtermann09} and also by Simonyi, Tardif, and Zsb\'an~\cite{simonyi2013colourful}. (The notation comes from the original definition that used ``multihomomorphisms'' from $K_2$ to $G$.)

The neighborhood complex $\N(G)$ is defined as follows:
\[
\N(G) = \big\{ A \subseteq V \colon \CN(A) \neq \varnothing \big\}\, .
\]
Note that, contrary to the other complexes introduced so far, $\Z_2$ does not act on it.

\begin{figure}
    \centering
    \subcaptionbox{The cycle $C_4$}[0.4\textwidth]{
    \centering
\begin{tikzpicture}[scale = 2]
    \node[circle,fill,inner sep=1pt](a) at (0,0) {}  ;
    \node [minimum size=0.2cm]  (la) at (-0.1,-0.1) {\footnotesize $4$};
    \node[circle,fill,inner sep=1pt](b) at (1,0) {}  ;
    \node [minimum size=0.2cm]  (lb) at (1.1,-0.1) {\footnotesize $3$};
    \node[circle,fill,inner sep=1pt](c) at (1,1) {}  ;
    \node [minimum size=0.2cm]  (lc) at (1.1,1.1) {\footnotesize $2$};
    \node[circle,fill,inner sep=1pt](d) at (0,1) {}  ;
    \node [minimum size=0.2cm]  (ld) at (-0.1,1.1) {\footnotesize $1$};
    \draw (a)--(b)--(c)--(d)--(a);
\end{tikzpicture}}
\subcaptionbox{The box complex $\B_0(C_4)$}[0.4\textwidth]{
\begin{tikzpicture}[scale = 0.8]
\centering
\coordinate (W) at (4,0,0);
\node[circle,fill,inner sep=1pt](w) at (4,0,0) {};
\node [minimum size=0.2cm]  (lW) at (4.6,0,0) {\footnotesize $+1$};
\coordinate (X) at (0,0,0);
\node[circle,fill,inner sep=1pt](x) at (0,0,0) {};
\node [minimum size=0.2cm]  (lX) at (-0.6,0,0) {\footnotesize $-2$};
\coordinate (Y) at (2,0,3);
\node [minimum size=0.2cm]  (lY) at (2.2,0,3.6) {\footnotesize $-4$};
\node[circle,fill,inner sep=1pt](y) at (2,0,3) {};
\coordinate (Z) at (2,2,1);
\node[circle,fill,inner sep=1pt](z) at (2,2,1) {};

\coordinate (A) at (-1.3,4,0);
\node [minimum size=0.2cm]  (lA) at (-1.9,4,0) {\footnotesize $-1$};
\node[circle,fill,inner sep=1pt](a) at (-1.3,4,0) {};
\coordinate (B) at (2.7,4,0);
\node [minimum size=0.2cm]  (lB) at (3.3,4,0) {\footnotesize $+2$};
\node[circle,fill,inner sep=1pt](b) at (2.7,4,0) {};
\coordinate (C) at (0.7,4,-3);
\node[circle,fill,inner sep=1pt](c) at (0.7,4,-3) {};
\node [minimum size=0.2cm]  (lC) at (0.5,4,-3.6) {\footnotesize $+4$};
\coordinate (D) at (0.7,2,-1);
\node[circle,fill,inner sep=1pt](d) at (0.7,2,-1) {};



\draw[fill=lightgray,fill opacity=.7] (W)--(X)--(Y)--cycle;
\draw[fill=lightgray,fill opacity=.7] (W)--(X)--(Z)--cycle;
\draw[fill=lightgray,fill opacity=.7] (W)--(Z)--(Y)--cycle;
\draw[fill=lightgray,fill opacity=.7] (Z)--(X)--(Y)--cycle;


\draw[fill=lightgray,fill opacity=.7] (W)--(Z)--(C)--cycle;
\draw[fill=lightgray,fill opacity=.7] (C)--(B)--(W)--cycle;
\draw[fill=lightgray,fill opacity=.7] (B)--(C)--(Z)--cycle;
\draw[fill=lightgray,fill opacity=.7] (W)--(Z)--(B)--cycle;

\draw[fill=lightgray,fill opacity=.7] (A)--(D)--(C)--cycle;

\draw[fill=lightgray,fill opacity=.7] (D)--(B)--(C)--cycle;
\draw[fill=lightgray,fill opacity=.7] (A)--(B)--(C)--cycle;
 \draw[fill=lightgray,fill opacity=.7] (A)--(D)--(B)--cycle;


\draw [fill=lightgray,fill opacity=.7](X)--(D)--(A)--cycle;
\draw[fill=lightgray, opacity=0.7,] (X)--(D)--(Y)--cycle;
\draw[fill=lightgray,fill opacity=.7] (X)--(Y)--(A)--cycle;
\draw [fill=lightgray,fill opacity=.7](Y)--(D)--(A)--cycle;

\node [minimum size=0.2cm]  (lD) at (0.25,1.8,-1.2) {\footnotesize $-3$};\node [minimum size=0.2cm]  (lZ) at (2.5,2.2,1.2) {\footnotesize $+3$};



\end{tikzpicture}}
    
\caption{\label{fig:my_label} The $4$-cycle and its box complex $\B_0(C_4)$}
\end{figure}
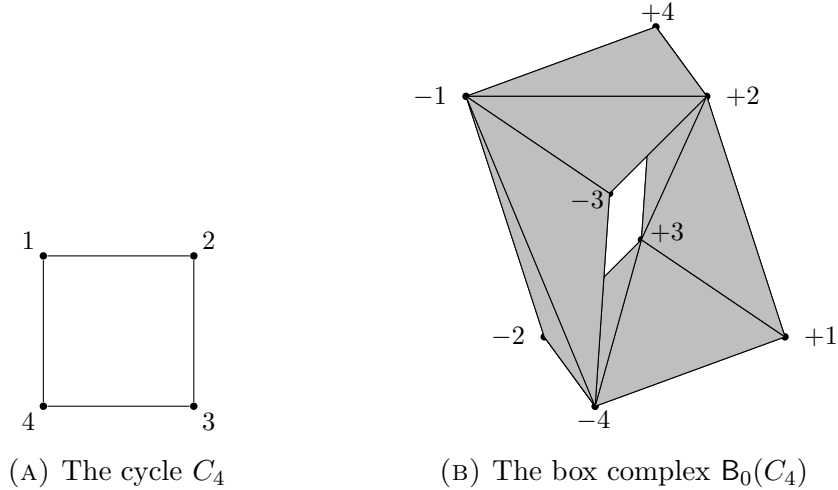

The following theorems show that these four complexes are strongly related. The next theorem is an immediate consequence of results due to Csorba et al.~\cite{csorba2004box}.

\begin{theorem}[Csorba et al.~\cite{csorba2004box}]\label{thm:NB}
For every graph $G$, the complexes $\B(G)$ and $\N(G)$ are homotopy equivalent.
\end{theorem}

\begin{example}
    The neighborhood complex $\N(K_n)$ of the complete graph with $n$ vertices is the boundary of the $(n-1)$-dimensional simplex. It is thus homotopy equivalent to $S^{n-2}$, to which $\B(K_n)$ is also homotopy equivalent (Example~\ref{ex:box-complete}).
\end{example}

Combining this theorem with the inequality~\eqref{eq:bg} leads to the inequality 
\begin{equation}\label{eq:ng}
\conn(\N(G))+3 \leq \chi(G)\, .
\end{equation}
Historically speaking, it is the first topological lower bound on the chromatic number and it is due to Lov\'asz.

There is a topological relation between $\B(G)$ and $\Hom(K_2,G)$. A poset $P$ is brought in the realm of topology via the \emph{order complex}, which is the simplicial complex whose ground set is formed by the elements of $P$ and whose simplices are the collections of pairwise comparable elements. It is denoted by $\triangle(P)$.

\begin{theorem}[\v{Z}ivaljevi\'c~\cite{vzivaljevic2005wi}]\label{thm:BH}
For every graph $G$, the complexes $\B(G)$ and ${\triangle}\big( \Hom(K_2,G) \big)$ are $\Z_2$-homotopy equivalent.
\end{theorem}

\subsection{Box complex of some specific graphs}\label{subsec:spec-graph}

Kneser's conjecture, whose resolution by Lov\'asz is generally considered as the birth of topological combinatorics, was about the chromatic number of Kneser graphs. These graphs, which are important objects from discrete mathematics, continue to play a central role in topological combinatorics. Given two integer numbers $n$ and $k$, with $n \geq 2k-1$, the {\em Kneser graph} $\KG(n,k)$ has all $k$-subsets of $[n]$ as vertex set and has an edge between each pair of disjoint vertices.

 To achieve the computation of the chromatic number of Kneser graphs, Lov\'asz proved first inequality~\eqref{eq:ng}, and then that $\conn(\N(\KG(n,k)))$ is equal to $n-2k-1$. Since $n-2k+2$ is an easy upper bound, the equality $\chi(\KG(n,k)) = n-2k+2$ follows immediately. This result about the neighborhood complex of Kneser graphs has been improved in various ways. The {\em wedge} of topological spaces is a standard operation in topology and consists in joining the topological spaces at a single point.

\begin{theorem}[De Longueville~{\cite[Proposition 2.4]{de2012course}}]
The box complex $\B(\KG(n,k))$ is homotopy equivalent to the wedge of an odd number of $(n-2k)$-dimensional spheres.
\end{theorem}

Schrijver graphs were introduced by Schrijver~\cite{schrijver1978vertex} soon after the resolution of Kneser's conjecture. Their role in topological combinatorics is also important. Given two integer numbers $n$ and $k$, with $n \geq 2k-1$, the {\em Schrijver graph} $\SG(n,k)$ is the subgraph of $\KG(n,k)$ induced by the {\em $2$-stable} $k$-subsets of $[n]$, which are the $k$-subsets not containing consecutive integers (with $1$ and $n$ being considered as consecutive). Schrijver proved that the chromatic number of $\SG(n,k)$ is $n-2k+2$ as well (but that they are {\em critical}, in the sense of the removal of any vertex reducing the chromatic number). Bj\"orner and De Longueville~\cite{bjorner2003neighborhood} proved in 2003 that the neighborhood complex of $\SG(n,k)$ is homotopy equivalent to an $(n-2k)$-dimensional sphere. Combined with inequality~\eqref{eq:ng}, this provides an alternative proof of $\chi(\SG(n,k))\geq n-2k+2$.
The next theorem is a strengthening of their result. 

\begin{theorem}[Schultz~\cite{schultz2011equivariant}]\label{thm:schultz}
The box complex $\B(\SG(n,k))$ is $\Z_2$-homotopy equivalent to an $(n-2k)$-dimensional sphere.
\end{theorem}

More generally, {\em $s$-stable $k$-subsets} of $[n]$ are defined as those $k$-subsets $X \subseteq [n]$ such that $s \leq |j-i| \leq n-s$ for every pair $i,j$ of distinct elements of $X$. For $n \geq sk$, the chromatic number of the {\em $s$-stable Kneser graph}, defined as the subgraph of $\KG(n,k)$ induced by the $s$-stable $k$-subsets of $[n]$, has been conjectured to be equal to $n-sk$~\cite{meunier2011chromatic}. When $s \geq 4$, Jonsson~\cite{jonsson2012chromatic} verified this conjecture for large $n$ in terms of $s$ and $k$. The conjecture has also been proved for all even $s$ by Chen~\cite{chen2015multichromatic} and for $k=2$ by the authors and Mizrahi~\cite{daneshpajouh2021colorings}. See also Frick~\cite{frick2020chromatic} for complementary results. Studying the box complex of these graphs could form a reasonable approach to this conjecture. Yet, this is still an almost unexplored field (see Oszt\'enyi~\cite{osztenyi2019neighborhood} and Daneshpajouh and Oszt\'enyi~\cite{daneshpajouh2021neighborhood} for related results).\label{page:sstable}


Here is a result about another classical family of graphs.

\begin{theorem}[Csorba~\cite{csorba2010homotopy}]
The box complex $\B(G)$ of a connected chordal graph $G$ is homotopy equivalent to the wedge of an odd number of spheres.
\end{theorem}

\section{Parameters of graphs and complexes}\label{sec:parameter}



This section is devoted to the definition of the parameters present in Figure~\ref{fig}. A subsection concerns combinatorial parameters and another concerns topological parameters. Most of these parameters are then compared in a subsequent subsection. Two very important parameters are not defined hereafter since we have already met them in Section~\ref{sec:first}: the connectivity of a simplicial complex or a topological space, denoted by $\conn(\cdot)$, and the chromatic number of a graph, denoted by $\chi(\cdot)$.

\subsection{Combinatorial parameters}\label{subsec:comb}

The most classical lower bound on the chromatic number is the {\em clique number}, defined as the maximal number of vertices of a {\em clique}, i.e., a complete subgraph. It is denoted by $\omega(G)$. Another relevant parameter is the largest $n$ such that $G$ contains a biclique with parts of sizes $k,\ell$ for all $k,\ell \geq 1$ satisfying $k+\ell = n$. (We remind that {\em bicliques}, already met in Section~\ref{sec:first}, are complete bipartite subgraphs.) It is denoted by $b(G)$.


Consider a graph $G$ with a proper coloring $c \colon V(G) \rightarrow \Z_+$. A biclique is {\em zigzag} if there exists a numbering $v_1, v_2, \ldots, v_t$ of its vertices such that the vertices with an odd index form one part of the biclique, those with an even index form the other part, and 
\[
c(v_1) < c(v_2) < \cdots < c(v_t) \, .
\]
The notion of zigzag bicliques originates from the work of Simonyi and Tardos~\cite{simonyi2006local} (but the terminology is ours). The {\em zigzag number} of $G$, denoted by $\zig(G)$, is the minimum over the proper colorings of $G$ of the size of a largest zigzag biclique. We clearly have $\zig(G) \leq \chi(G)$.

A hypergraph is {\em $2$-colorable} if its vertices can be colored with two colors in such a way that no edge is monochromatic. The {\em $2$-colorability defect} of a hypergraph $\HH$, denoted by $\cd(\HH)$, is the minimal number of vertices to remove from $\HH$ such that the hypergraph induced by the remaining vertices is $2$-colorable:
\[
\cd(\HH)=\min\left\{|U|\colon \left(V(\HH)\setminus U,\{e\in E(\HH)\colon\;e\cap U=\varnothing\}\right)\mbox{ is $2$-colorable}\right\}\, .
\]
This parameter has been introduced by Dol'nikov~\cite{dol1988certain} who proved that it also provides a lower bound on the chromatic number with the help of the notion of Kneser representation. A {\em Kneser representation} of a graph $G$ is a hypergraph $\HH$ with a one-to-one mapping from $V(G)$ to $E(\HH)$ such that any two vertices are adjacent if and only if they are mapped to disjoint edges of $\HH$. Given a Kneser representation $\HH$ of $G$, then Dol'nikov's inequality reads
\begin{equation}\label{eq:dol}
\cd(\HH) \leq \chi(G)\, .
\end{equation}

\begin{example}
For the Kneser graph $\KG(n,k)$, the complete $k$-uniform hypergraph on $[n]$, which we denote by $\mathcal{K}_n^k$, is a natural Kneser representation. Removing less than $n-2k+2$ vertices from that hypergraph does not make it $2$-colorable: any $2$-coloring will color with the same color at least $k$ elements. Removing exactly $n-2k+2$ vertices does make it $2$-colorable: color $k-1$ elements with one color; color the other $k-1$ elements with the other color. Thus $\cd(\mathcal{K}_n^k)=n-2k+2$, and we get $\chi(\KG(n,k))\geq n-2k+2$. We already mentioned that $n-2k+2$ is an easy upper bound, and so inequality~\eqref{eq:dol} is tight for Kneser graphs.
\end{example}

It is actually not difficult to come up with a Kneser representation of any graph $G$. The standard construction consists in representing each vertex of $G$ by the set of its incident edges in the complement graph; in some cases, unfortunately, vertices can be represented by the same set, an issue that can be easily fixed by introducing an extra element for each vertex. The question of finding a Kneser representation of minimal size has been investigated by Hamburger, Por, and Walsh~\cite{hamburger2009kneser}.

The last combinatorial parameters we define are related to posets. A {\em chain} in a poset is a collection of pairwise comparable elements. The first parameter is the {\em height} of a poset $P$, denoted by $\height(P)$, and defined as the maximum size of a chain. It is a classical parameter. The second one is less common and has been introduced by Simonyi, Tardif, and Zsb\'an~\cite{simonyi2013colourful} in relations with topological bounds on the chromatic number. A {\em $\Z_2$-poset} is a poset endowed with an order-preserving involution. The group $\Z_2$ is thus acting on such a poset. For a free $\Z_2$-poset $P$, we define its {\em cross-index}, denoted by $\Xind(P)$, as the minimum $t$ such that there exists a $\Z_2$-map from $P$ to $Q_t$, where $Q_t$ is a $\Z_2$-poset defined as follows. It has $\{\pm 1,\pm 2,\ldots, \pm (t+1)\}$ as ground set, $x\preceq y$ for $x,y\in Q_t$ if $|x|\leq |y|$, and whose order-preserving involution is the map $x\mapsto -x$. We have 
\begin{equation}\label{eq:hP}
    \Xind(P) \leq \height(P) -1 \, .
\end{equation}
Indeed, assign a sign $+$ to one element and a sign $-$ to the other element of every orbit; define a $\Z_2$-map $\phi$ from $P$ to $Q_{\height(P)-1}$ by setting $\phi(x)$ to be the largest cardinality of a chain ending at $x$ with the sign assigned to $x$.

\subsection{Topological parameters}\label{subsec:topol} From now on, we assume basic knowledge in algebraic topology. For the terminology and the notation not defined in the present paper, we refer the reader to the book by Hatcher~\cite{hatcher2005algebraic}.

The {\em index} of a topological $\Z_2$-space $X$ is the smallest $d$ such that there exists a continuous $\Z_2$-map from $X$ to $S^d$ (the central symmetry making this latter a $\Z_2$-space). By convention, it is equal to $+\infty$ when there is no such map for any $d$, and to $-1$ when $X$ is empty. The {\em coindex} of a topological $\Z_2$-space $X$ is the largest $d$ such that there exists a continuous $\Z_2$-map from $S^d$ to $X$. By convention, it is equal to $-1$ when $X$ is empty and to $+\infty$ when there is such a map for arbitrarily large $d$. (Note that when $X$ is non-empty, such a map always exists for $d=0$.) Therefore, both the index and the coindex take their values in $\{-1,0,1,2,\ldots\} \cup \{+\infty\}$. The Borsuk--Ulam theorem implies (and is actually equivalent to) the inequality $\coind(X) \leq \ind(X)$ for every topological $\Z_2$-space $X$.

We already defined the notion of free simplicial $\Z_2$-complexes. Similarly, a {\em free} $\Z_2$-space is such that every orbit is of size $2$. It is immediate to check that index and coindex of non-free topological $\Z_2$-spaces are equal to $+\infty$. So, even if the definition of index and coindex is valid for non-free topological $\Z_2$-spaces, it is only relevant for free topological $\Z_2$-spaces.


A topological parameter that is closely related to the connectivity is the $\Z_2$-connectivity. The \emph{$\Z_2$-connectivity} of a topological space $X$ is denoted by $\conn_{\Z_2}(X)$ and is the maximal integer $d$ such that the reduced homology groups $\widetilde H_k(X,\Z_2)$ with $\Z_2$ coefficients are trivial for all $k \in\{-1,0,1,\ldots,d\}$. The classical Hurewicz theorem states in particular that the inequality $\conn(X)\leq\conn_{\Z_2}(X)$ holds for every such topological space $X$.

The last topological parameter we consider is also maybe the most abstract. We define it for a (finite) simplicial free $\Z_2$-complex $\K$. There always exists a $\Z_2$-map from such a $\K$ to the infinite-dimensional sphere $S^{\infty}$. This can be easily achieved by induction on the dimension of $\K$, the sphere $S^{\infty}$ being contractible; see~\cite[Proposition 5.3.2(v)]{matousek2008using} or \cite[Proposition 8.16]{kozlov2008combinatorial} for a complete proof. Pick any such map $f$. This $\Z_2$-map $f$ induces a map $\bar f \colon \K / \Z_2 \rightarrow S^{\infty} / \Z_2 = \R P^{\infty}$ and in cohomology we get a map $\bar f^* \colon H^*(\R P^{\infty},\Z_2) \rightarrow H^*(\K/\Z_2,\Z_2)$. The graded algebra $H^*(\R P^{\infty},\Z_2)$ is of the form $\Z_2[\alpha]$, with generator $\alpha$ taken in $H^1(\R P^{\infty},\Z_2)$~\cite[Theorem 3.19]{hatcher2005algebraic}. Denoting by $\varpi_1(\K)$ the image of $\alpha$ by $\bar f^*$, the {\em cohomological-index} of $\K$, also called its \emph{Stiefel--Withney height}, is the maximum $n$ such that $(\varpi_1(\K))^n \in H^n(\K/\Z_2,\Z_2)$ is not $0$. (Here, the power $n$ is computed according to the cup product of the cohomology ring $H^*(\K/\Z_2,\Z_2)$.) We denote it by $\hind(\K)$. The cohomological index is well-defined because the map $\bar f^*$ does not depend on the choice for $f$: indeed, all $\Z_2$-maps from $\K$ to $S^{\infty}$ induce the same map $\K / \Z_2 \rightarrow S^{\infty} / \Z_2 = \R P^{\infty}$ up to homotopy; see~\cite[Theorem 8.17]{kozlov2008combinatorial}.

\subsection{General relations between various parameters}\label{subsec:gen-rel-par}


For a free (finite) simplicial $\Z_2$-complex $\K$, we always have

 \begin{tikzpicture}
        \node at (2/3,2/3) (s1)  {{$\labelrel\uleq{a1}$}};
        \node at (2/3,-1/3) (s2)  {{$\labelrel\dleq{a2}$}};
        \node at (10.7/3,2/3) (s3)  {{$\labelrel\dleq{b1}$}};
        \node at (11/3,-1/3) (s4)  {{$\labelrel\uleq{b2}$}};
        \node at (-2/3,0) (a1) {$1+\conn(\K)$};
        \node at (2.2,0.85) (a2) {$\coind(\K)$};
        \node at (2.3,-1) (a3) {$1+\conn_{\Z_2}(\K)$};
        \node at (5.5,0) (a4) {$\hind(\K)$};
        \node at (7.25,1/8) (s5)  {{$\labelrel\leq{c}$}};
        \node at (8.25,0) (a5) {$\ind(\K)$};
         \node at (9.25,1/8) (s6)  {{$\labelrel\leq{d}$}};
        \node at (10.75,0) (a6) {$\Xind(\F(\K))$};
        \node at (12.25,1/8) (s7)  {{$\labelrel\leq{e}$}};
        \node at (13.35,0) (a7) {$\dim(\K) \, ,$};
\end{tikzpicture}  
where $\F(\K)$ is the face poset of $\K$. These inequalities will be mostly used in Section~\ref{sec:comment}, where we explain Figure~\ref{fig} in detail. We now discuss briefly each inequality.

\begin{proof}[Inequality~\eqref{a1}] 
The inequality $1+\conn(\K) \leq \coind(\K)$ is proved for instance by Matou\v sek in his book~\cite[p.97, proof of (iv)]{matousek2002topological}. (The objective in this reference is to prove a lower bound on the index, but the proof is actually showing the stronger result with coindex.)
\end{proof}

\begin{proof}[Inequality \eqref{a2}]
It is a consequence of the Hurewicz theorem.
\end{proof}

\begin{proof}[Inequality \eqref{b1}] 
It is a now classical inequality but we provide a complete and self-contained proof. 

We first establish that the cohomological index of a sphere is at least its dimension (it is actually equal; see discussion of inequality~\eqref{c} below). The conclusion will then follow easily. The graded algebra $H^*(\R P^d,\Z_2)$ is of the form $\Z_2[\beta]/\beta^{d+1}$ for some generator $\beta$ \cite[Theorem 3.19]{hatcher2005algebraic}. The inclusion map $S^d \xhookrightarrow{} S^{\infty}$ induces an injective homomorphism 
$H^*(\R P^{\infty},\Z_2) \rightarrow H^*(\R P^d,\Z_2)$ (this is because written as CW-complexes, $\R P^d$ is exactly the $d$-skeleton of $\R P^{\infty}$; see~\cite[Lemma 2.34, (c)]{hatcher2005algebraic} for a proof for a similar property for homology). This injectivity implies in particular that $\varpi_1(S^d) = \beta$, and thus $(\varpi_1(S^d))^d \neq 0$. Hence, $\hind(S^d)$ is at least $d$. To conclude, assume the existence of a $\Z_2$-map $g\colon S^d \rightarrow \K$. The composition $\bar g^* \circ \bar f^*\colon H^*(\R P^{\infty},\Z_2) \rightarrow H^*(\R P^d,\Z_2)$, with $\bar f^*$ defined as in Section~\ref{subsec:topol}, implies that $(\varpi_1(S^d))^n=0$ if $(\varpi_1(\K))^n=0$. Hence, $(\varpi_1(\K))^d\neq 0$.
\end{proof}

\begin{proof}[Inequality \eqref{b2}]
This inequality requires some work, and we refer to the work by Conner and Floyd~\cite[4.7]{conner1960} (where we use the isomorphism between homology and cohomology for finite simplical complexes and coefficients in a field); see also~\cite[Corollary 3.3]{kozlov2006homology}.
\end{proof}

\begin{proof}[Inequality \eqref{c}]
The quantity $\hind(S^d)$ is at most $d$ (actually, is equal to $d$ because of the argument used for proving inequality~\eqref{b2}), since $H^n(S^d/\Z_2)=0$ when $n>d$. Following similar lines as in the proof of inequality~\eqref{b1}, we get the desired inequality.
\end{proof}

\begin{proof}[Inequality \eqref{d}]
Let $t = \Xind(\F(\K))$. Any $\Z_2$-map from $\F(\K)$ to $Q_t$ lifts to a $\Z_2$-map from $\triangle \F(\K)$ to $\triangle Q_t$. The simplicial complex $\triangle Q_t$ is $\Z_2$-homeomorphic to a $t$-dimensional sphere. This implies inequality~\eqref{d}.
\end{proof}
\begin{proof}[Inequality \eqref{e}]
The dimension of $\K$ is the same as the height of its face poset minus one; therefore, inequality~\eqref{e} is a specialization of inequality~\eqref{eq:hP} for $P=\F(\K)$.
\end{proof}


\section{Computability and complexity of parameters}\label{sec:comput}

\subsection{Decidability and hardness}\label{subsec:dec-compl}

Clearly, all combinatorial parameters of Section~\ref{subsec:comb} are decidable (i.e., given an integer $k$, there is an algorithm deciding whether the parameter is at most $k$). The same thing holds for topological parameters based on homological or cohomological quantities. For the other parameters, we do not know the answer, except for $\conn(\B(G))$ and $\conn(\B_0(G))$. \label{page:dec}The former is undecidable since, via Csorba's result, it is equivalent to the decidability of topological connectivity in general, which is not possible~\cite[Corollary 3.9]{davis1977unsolvable}. (We reduce to $\Z_2$-complexes just by taking the product with a sphere of sufficiently large dimension.) For the latter, which is a better bound, we have the following.

\begin{theorem}\label{thm:dec_B0}
The quantity $\conn(\B_0(G))$ is decidable. 
\end{theorem}

Actually, the proof will make clear that there is an algorithm realizing the decision task in a time that is polynomial in the size of $\B_0(G)$.

\begin{proof}[Proof of Theorem~\ref{thm:dec_B0}]
If $G$ has no vertex, then $\conn(\B_0(G))$ is $-2$ because $\B_0(G)$ has no vertices and no simplices.

If $G$ has at least one vertex and no edges, then $\conn(\B_0(G))$ is $-1$ because $\B_0(G)$ has at least two $0$-simplices (obtained from a same vertex in the graph) but no $1$-simplices.

For the remaining cases, we assume that $G$ has at least one edge, and thus that $\B_0(G)$ is path-connected because $\B(G)$ is non-empty and $\B_0(G)$ is the suspension of $\B(G)$ (Theorem~\ref{thm:susp}). Moreover, we will use the (well-known) fact that when $G$ is connected, $\B(G)$ is path-connected if and only if $G$ is not bipartite (see, e.g., ~\cite[p.~320]{lovasz1978kneser} where it is stated for the neighborhood complex, which is homotopy equivalent according to Theorem~\ref{thm:NB}).

If $G$ is bipartite or not connected, then $\conn(\B_0(G))$ is $0$ because in that case $\B(G)$ is not path-connected, which implies that the reduced $1$-dimensional homology of $\B_0(G)$ is non-trivial (consequence of Theorem~\ref{thm:susp} and equality~\eqref{eq:homo-susp}).

If $G$ is non-bipartite and connected, then $\conn(\B_0(G))$ is equal to its homological connectivity over $\Z$ because in that case $\B(G)$ is path-connected, which implies that $\B_0(G)$ is simply connected (by equality~\eqref{eq:conn-susp}) and thus that the Hurewicz theorem applies. Since homology over $\Z$ can be computed via elementary linear algebra, this finishes the proof.
\end{proof}

The computation of $\chi(G)$ and $\omega(G)$ are among the first problems that have been proved $\NP$-hard~\cite{karp1972reducibility}. The computation of $\cd(\HH)$ and $b(G)$ are $\NP$-hard as well. For the first parameter, see~\cite{alishahi2017strengthening}, and for the second, this can be shown by a direct reduction from the ``maximum balanced biclique,'' which is \NP-hard~\cite{garey}: the existence in $G$ of a biclique whose parts are both of size at least $k$ is equivalent to $b(G') \geq 2k$, where $G'$ is the disjoint union of $G$ and all possible bicliques with part sizes $\ell,m$, with $\ell+m=2k$ and $\ell\leq k-1$.

The complexity status of the other parameters is not known. Note that even if computing the homology or cohomology is polynomial in terms of the size of the simplicial complex, it does not imply any polynomiality result for graphs since the complexes we consider are exponential in the size of the graph. 

The parameters $\zig(G)$ and $\Xind(\Hom(K_2,G))$ are the combinatorial parameters for which the complexity status is still open. Their combinatorial essence make their complexity status a natural challenge. Simonyi, Tardif, and Zsb\'an~\cite{simonyi2013colourful} proved that computing $\Xind(P)$ for any $\Z_2$-poset $P$ is $\NP$-hard, but this leaves open the complexity of computing $\Xind(P)$ when $P$ is the Hom complex of a graph $G$. \label{page:compl}




\subsection{Borsuk--Ulam boundary}\label{subsec:BU} Except for the clique number, the Borsuk--Ulam theorem---or a re-proof of it---is used to prove $\lb(G) \leq \chi(G)$ for all lower bounds $\lb$ below the red line in Figure~\ref{fig}. Moreover, if we also discard the lower bound $\cd(\HH)$, the inequality $\lb(G) \leq \chi(G)$ can be used to establish the Borsuk--Ulam theorem, as we will discuss further below. Concerning the inequality $\cd(\HH) \leq \chi(G)$ (with $\HH$ a Kneser representation of $G$), we are not aware of any way to derive the Borsuk--Ulam theorem from it. This will also be discussed a bit further at the end of the section.

On the other hand, the proof of $\lb(G) \leq \chi(G)$ is trivial for all lower bounds $\lb$ above the red line, namely when
\[
\lb(G) \in \{\hind(\B(G))+2,\ind(\B_0(G))+1,\ind(\B(G))+2,\Xind(\Hom(K_2,G))+2,\zig(G)\} \, .
\]
For $\zig(G)$, it is immediate from the definition. For the other bounds, it is also immediate once we notice that any proper coloring of $G$ with $n$ colors translates directly into a $\Z_2$-map from $\B(G)$, $\B_0(G)$, and $\Hom(K_2,G)$ to respectively $\B(K_n)$, $\B_0(K_n)$, and $\Hom(K_2,K_n)$. These inequalities are ``functorial,'' since they are just translations of the existence of a map in a setting into another setting. It does not seem that any of these inequalities can be used to prove the Borsuk--Ulam theorem.

These are the reasons why we call the red line in Figure~\ref{fig} the Borsuk--Ulam boundary: the topological lower bounds above the line does not require the Borsuk--Ulam theorem; the topological lower bounds below require the Borsuk--Ulam theorem (apart possibly for $\cd(\HH)$).

We are fully aware that this kind of statements is hard to support with mathematical arguments. Yet, we provide now an elementary proof of the Borsuk--Ulam theorem from the inequality $\conn(\B(G))+3 \leq \chi(G)$, and a complexity result showing that the statement $\coind(\B(G))+2 \leq \chi(G)$ is somehow equivalent to the Borsuk--Ulam theorem. We finish with a computational problem related to Schrijver's graph, which also illustrates the fact all topological lower bounds below the dashed line in Figure~\ref{fig} require the Borsuk--Ulam theorem (apart possibly for $\cd(\HH)$).

\subsubsection*{Proving the Borsuk--Ulam theorem from $\conn(\B(G))+3 \leq \chi(G)$}

The {\em Borsuk graph} $\Bor^d(\varepsilon)$ has the points of the unit sphere $S^d$ as vertices, and has an edge between two vertices if they are at distance at least $2-\varepsilon$ apart (Euclidean distance). It is well-known that the Borsuk--Ulam theorem is an immediate consequence of the chromatic number of $\Bor^d(\varepsilon)$ being at least $d+2$; see, e.g., Exercise 10 in \cite[section 2.2]{matousek2008using}. Let $r$ be an arbitrary positive integer. From the simplicial complex $\sd^r(\partial\Diamond^{d+1})$ ($r$-th barycentric subdivision of the boundary of the $(d+1)$-dimensional cross-polytope), we construct a graph $G$ , as done in Csorba's proof of Theorem~\ref{thm:csorba} (see Section~\ref{subsec:univers}). This proof ensures that the simplicial complexes $\B(G)$ and $\sd^r(\partial\Diamond^{d+1})$ are $\Z_2$-homotopy equivalent, and thus that $\conn(\B(G)) = d-1$ holds. The inequality $\conn(\B(G))+3 \leq \chi(G)$ implies then that
$d+2 \leq \chi(G)$. Since $G$ is actually isomorphic to a subgraph of $\Bor^d(\varepsilon)$ when $r$ is a sufficiently large integer, we have $\chi(\Bor^d(\varepsilon)) \geq d+2$, as required.

\subsubsection*{The coindex lower bound is polynomially equivalent to the Borsuk--Ulam theorem}
The problem \textsc{$\varepsilon$-Borsuk--Ulam} is defined as follows. An {\em arithmetic circuit} is a representation of a continuous function defined by a few elementary operations (addition, multiplication, maximum, etc.); see, e.g., ~\cite{deligkas2021computing}.

\medskip

\noindent \textsc{$\varepsilon$-Borsuk--Ulam}:

{\bf Input.} An arithmetic circuit $f\colon \R^{d+1} \rightarrow \R^d$; a constant $\varepsilon\in\R$.

{\bf Task.} Find a point $x \in S^d$ such that $\|f(x) - f(-x)\|_{\infty} \leq \varepsilon$.

\medskip

The problem \textsc{$\varepsilon$-Borsuk--Ulam} is \PPA-complete~\cite{deligkas2021computing}. The complexity class \PPA{} is one of the fundamental subclasses of \TFNP, the \NP-search problems that always have a solution. It has been introduced, with other classes, by Papadimitriou in 1994~\cite{papadimitriou1994complexity} and is defined as those \TFNP{} problems that can be reduced to the problem of finding another odd-degree vertex in a graph which is given with a first odd-degree vertex. \textsc{$\varepsilon$-Borsuk--Ulam} being \PPA-complete means, as for \NP-complete problems, that it is unlikely that the problem can be solved in polynomial time. (Note that \textsc{$\varepsilon$-Borsuk--Ulam} is not in \NP{} since it is not a decision problem.) Now, we define a computational problem related to the inequality $\coind(\B(G))+2 \leq \chi(G)$. A {\em $\Z_2$-circuit} $g$ is a circuit such that $g(-x)=-g(x)$ for all $x$ in the domain of $g$ (this property can be enforced syntactically).

\medskip

\noindent\textsc{Coind-lower-bound}:

{\bf Input.} A graph $G$, along with a (non-necessarily proper) coloring using at most $d+1$ distinct colors; an arithmetic $\Z_2$-circuit $g\colon\R^{d+1}\rightarrow \R^{V(G)}$; a constant $\delta\in\R$.

{\bf Task.} Find one of the following.
\begin{enumerate}
\item\label{item:zero} A point $x \in S^d$ such that $\|g(x)\|_{\infty} \leq \delta$.
\item\label{item:non-embed} A point $x \in S^d$ such that 
\[
\left\{s \cdot v\colon s \in \{+,-\},\,  v \in V(G), \,  s g_v(x) > \frac 1 {2(|V(G)|+1)}\delta\right\}
\] is not a simplex of $\B(G)$.
\item\label{item:monoc} A monochromatic edge of $G$.
\end{enumerate}

\medskip

The rationale of the problem is the following. It considers implicitly that $\B(G)$ is embedded with the two images of the orbit of each vertex $v$ located at $e_v$ and $-e_v$ (the $e_v$ being the unit vectors of the standard basis of $\R^{V(G)}$). If the coloring is proper, the inequality $\coind(\B(G))+2 \leq \chi(G)$ implies that there is no continuous $\Z_2$-map $S^d\rightarrow \B(G)$. Thus, any continuous $\Z_2$-map $\R^{d+1}\rightarrow\R^{V(G)}$ (as $g$) does not map $S^d$ to $\B(G)$ (expressed in two ways via type-\eqref{item:zero} solutions and type-\eqref{item:non-embed} solutions). Working with arithmetic circuit imposes some technicalities.

\begin{theorem}\label{thm:complex-coind}
    \textsc{Coind-lower-bound} is polynomially equivalent to \textsc{$\varepsilon$-Borsuk--Ulam}.
\end{theorem}

Since the proof is a bit long and technical, it is given in the appendix.

\subsubsection*{Schrijver's graphs computationally}
Schrijver's graphs $\SG(n,k)$ have been introduced in Section~\ref{subsec:spec-graph}. The problem \textsc{Schrijver} is defined as follows.  A {\em Boolean circuit} is a representation of a Boolean function defined by a few elementary operations (AND, OR, NOT).

\medskip

\noindent \textsc{Schrijver}:

{\bf Input.} Two integers $n$ and $k$ such that $n\geq 2k-1$; a Boolean circuit 
\[
c\colon \big\{2\text{-stable $k$-subsets of $[n]$}\big\} \rightarrow [n-2k+1]
\]
representing a coloring of $\SG(n,k)$.

{\bf Task.} Find a monochromatic edge of $\SG(n,k)$, i.e., two disjoint $2$-stable $k$-subsets $S,T$ of $[n]$ such that $c(S)=c(T)$.

\medskip

Since $\chi(\SG(n,k))=n-2k+2$, the coloring $c$ cannot be proper and a monochromatic edge always exists. Haviv proved the following result, which means that \textsc{$\varepsilon$-Borsuk--Ulam} can be polynomially reduced from and to \textsc{Schrijver}. The proof makes these reductions explicit.

\begin{theorem}[Haviv~\cite{haviv2022complexity}]
\textsc{Schrijver} is \PPA-complete.
\end{theorem}

The inequality $\chi(\SG(n,k))\geq n-2k+2$ can be established without too much work from all topological lower bounds below the dashed line of Figure~\ref{fig}, except from $\cd(\HH)$, which provides a lower bound of $n-4k+4$ when $\HH$ is the natural Kneser representation ($[n]$ being its vertex set, and the $2$-stable $k$-subsets being the edges). This again shows that these topological lower bounds somehow implies the Borsuk--Ulam theorem.

Regarding the colorability defect, we note that $\cd(\HH)$ is at most $n-2k+1$ for every Kneser representation $\HH$ of $\SG(n,k)$, at least when $n$ is odd: indeed Corollary 2 in~\cite{alishahi2017strengthening} shows that when the colorability defect lower bound is tight for some Kneser representation, then the circular chromatic number is equal to the usual chromatic number, which is not the case for Schrijver graphs with odd $n$. This might indicate that the colorability defect lower bound is ``weaker'' than the Borsuk--Ulam theorem, and may not require this latter theorem to be established.

\textsc{Schrijver} can be generalized by using a number $m$ of colors that can be smaller than $n-2k+1$. When $m$ is sufficiently small (of order $O(n/k)$), the problem becomes polynomial~\cite{haviv2023finding}. The maximal $m$, as a function of $n$ and $k$, for which the problem remains polynomial is not known. Finding this value would help identify the level of lower bounds for which the Borsuk--Ulam theorem is essential.

Another computational problem that can be considered is \textsc{Kneser}, defined just by replacing in \textsc{Schrijver} the domain of the Boolean circuit $c$ by the full collection of $k$-subsets of $[n]$. The following problem is probably among the most important problems at the intersection of topological combinatorics and complexity theory.

\begin{question}\label{q:kneser}
    Is \textsc{Kneser} \PPA-complete?
\end{question}

 A negative answer would imply that more elementary proofs of the Lov\'asz theorem and of the colorability defect lower bound than those known could be possible.

\section{Categorical product of graphs and topological spaces}\label{sec:product}

Given two graphs $G$ and $H$, their {\em categorical product} $G\times H$ is the graph whose vertices is the set $V(G) \times V(H)$ and for which $(u,v)(u',v')$ forms an edge whenever $uu'$ is an edge of $G$ and $vv'$ is an edge of $H$. This product has attracted attention especially because it appears in the famous Hedetniemi conjecture~\cite{hedetniemi1966homomorphisms} stating that $\chi(G\times H)$ is equal to the minimum of $\chi(G)$ and $\chi(H)$. This conjecture has recently been disproved by Shitov~\cite{shitov2019counterexamples}, but it is known to hold for some graphs for which topological bounds are tight, as we discuss now.

Consider all graphs for which a given topological lower bound is tight. We may then ask whether they satisfy Hedetniemi conjecture. The answer is known to be `yes' for some topological lower bounds~\cite{Dochtermann09,hajiabolhassan2016hedetniemi,hell1979introduction, simonyi2010topological}. The most general result in this line is the following theorem by Daneshpajouh, Karasev, and Volovikov~\cite{daneshpajouh2023hedetniemi} (immediate consequence of Corollary 2.6 of their paper), which implies all previous results of this kind because the corresponding topological lower bounds lie below the cohomological index lower bound in Figure~\ref{fig}.

\begin{theorem}[Daneshpajouh, Karasev, and Volovikov~\cite{daneshpajouh2023hedetniemi}]\label{thm:cohom-hedet}
Consider two graphs $G$ and $H$ for which the cohomological index lower bound is tight. Then $\chi(G \times H) = \min\{\chi(G),\chi(H)\}$.
\end{theorem}

The proof of Theorem~\ref{thm:cohom-hedet} relies on the following result, which, according to Simonyi and Zsb\'an~\cite[Remark 3]{simonyi2010topological}, is a personal communication of Csorba.

\begin{theorem}[{Csorba}]\label{thm:cat-prod}
For every pair of graphs $G,H$, the complexes $\B(G \times H)$ and $\B(G) \times \B(H)$ are $\Z_2$-homotopy equivalent.
\end{theorem}

(The product $\B(G) \times \B(H)$ can be interpreted as a CW-complex, whose cells are the products of the simplices of $\B(G)$ and $\B(H)$, or simply as the Cartesian product of $\B(G)$ and $\B(H)$ seen as topological spaces, since only homotopy equivalence is at stake.) A similar relation does not hold for $\B_0(\cdot)$, as elementary examples show easily.

\begin{example}
    Consider the case where $G$ is an edge and $H$ a triangle. Then $G \times H$ is the $1$-skeleton of a prism. The complexes $\B(G)$ and $\B(H)$ are $\Z_2$-homotopy equivalent respectively to $S^0$ and $S^1$ (Example~\ref{ex:box-complete}). Since $G \times H$ is $C_6$, the complex of $\B(G\times H)$ is $\Z_2$-homeomorphic to the disjoint union of two copies of $S^1 \times [0,1]$ (Example~\ref{ex:box-C2n}, the stronger $\Z_2$-homeomorphism being immediate), which is indeed $\Z_2$-homotopy equivalent to the product of $S^0$ and $S^1$.
\end{example}

The proof of Theorem~\ref{thm:cohom-hedet} also relies on a Hedetniemi-like relation for the cohomological index of $\Z_2$-spaces. Most previous results on lower bounds below the cohomological index lower bound in Figure~\ref{fig} also required similar relations. Here are all known relations of this type:
\begin{itemize}
    \item For every pair $X,Y$ of topological spaces, we have 
    \[
    \conn(X \times Y) = \min\{\conn(X),\conn(Y)\}
    \]
    (well-known).
    \item For every pair $X,Y$ of topological spaces, we have 
    \[
    \conn_{\Z_2}(X \times Y) = \min\{\conn_{\Z_2}(X),\conn_{\Z_2}(Y)\}
    \]
    (K\"unneth's formula).
    \item For every pair of free $\Z_2$-spaces $X,Y$, we have 
    \[
    \coind(X \times Y) = \min\{\coind(X),\coind(Y)\}
    \]
    (easy to establish).
    \item For every pair $X,Y$ of free $\Z_2$-spaces that admit $\Z_2$-triangulations, we have 
    \[
    \hind(X \times Y) = \min\{\hind(X),\hind(Y)\}
    \]
    (proved by Daneshpajouh, Karasev, and Volovikov~\cite{daneshpajouh2023hedetniemi}).
\end{itemize}
Whether similar relations hold for the index parameter or for the cross-index parameter is open (see ~\cite{bui2023topological, matsushita2019index,wrochna2019inverse}).\label{page:hedet-param-top} A closely related open question is the following.

\begin{question}\label{q:hedet}
For which parameters $\param$ of Figure~\ref{fig}, do we have
\[
\param(G\times H) = \min\{\param(G),\param(H)\}\quad ?
\]
\end{question}
From the previous discussion (except for the clique number, but it is then immediate), we have this Hedetniemi-like relation when 
\[
\param(\cdot)\in\{\omega(\cdot),\coind(\B(\cdot)),\conn(\B(\cdot)),\conn_{\Z_2}(\B(\cdot)),\hind(\B(\cdot))\}\, .
\]
But this is open for all other bounds in Figure~\ref{fig}.\label{page:hedet-paramG}

\section{Joins}\label{sec:join}

Join is an important operation in topology, and it is especially useful for the study of topological bounds. 

\subsection{Preliminaries on the join and suspension operations in topology}

The {\em join} of two simplicial complexes $\K$ and $\L$, denoted by $\K * \L$, is the simplicial complex whose vertices are the elements of $V(\K) \uplus V(\L)$ and whose simplices are all $A \uplus B$ with $A \in \K$ and $B \in \L$. Remark that the join operation is commutative and associative (up to renaming the vertices). Since $\K * S^0$ is homeomorphic to the suspension of $\K$ (see Section~\ref{subsec:another} for the definition of the suspension of a topological space), the simplicial complex $\K * S^0$ is also called the suspension of $\K$ and denoted by $\susp(\K)$.

The join operation behaves well with respect to many parameters.
Regarding topological connectivity, we have the following relation, originally proved by Milnor~\cite{milnor1956construction}:
\begin{equation}\label{eq:conn}
\conn(\K * \L) \geq \conn(\K) + \conn(\L) +2 \, .
\end{equation}
In the case of $\Z_2$-connectivity, the K\"unneth formula provides directly (we are working with coefficients in a field):
\begin{equation}\label{eq:connZ2}
\conn_{\Z_2}(\K * \L) = \conn_{\Z_2}(\K) + \conn_{\Z_2}(\L) +2 \, .
\end{equation}

The special cases when $\L$ is $S^0$---the suspension operation---is useful in many situations, as this can already be seen by the present paper. For instance, equality~\eqref{eq:conn} implies that
\begin{equation}\label{eq:conn-susp}
\conn(\susp(\K)) \geq \conn(\K) + 1\, .
\end{equation}
In the case of homological connectivity for coefficients in a field, we have equality; this is a direct consequence of equality~\eqref{eq:connZ2} in the case of $\Z_2$. We have even the following equality (see~\cite[Exercise 32]{hatcher2005algebraic}):
\begin{equation}\label{eq:homo-susp}
\widetilde{H}_{i+1}(\susp(\K),\Z) = \widetilde{H}_i(\K,\Z) \, .
\end{equation}

We also have equality in case of the cohomological index~\cite{daneshpajouh2023hedetniemi}:
\begin{equation}\label{eq:hind}
\hind(\K * \L) = \hind(\K) + \hind(\L) +1 \, .
\end{equation}

For the index and coindex, the following is immediate:
\begin{equation}\label{eq:coind}
\coind(\K * \L) \geq \coind(\K) + \coind(\L) +1 \qquad\mbox{and} \qquad \ind(\K * \L) \leq \ind(\K) + \ind(\L) +1 \, .
\end{equation}

We finish this section by a short discussion on inequality~\eqref{eq:conn-susp} because it is sometimes misstated as an equality in the combinatorial literature (and even the more general inequality~\eqref{eq:conn} is sometimes stated erroneously as an equality). The gap can be arbitrarily large. We state it and provide a full proof.

\begin{proposition}\label{prop:susp}
For every integer $n\geq 3$, there is an $n$-dimensional simplicial complex $\K_n$ such that $\conn(\susp(\K_n))\geq n$ and $\conn(\K_n) = 0$. 
\end{proposition}

\begin{proof}
According to \cite{kervaire1969smooth, ratcliffe2005some}, for every $n\geq 3$, there is a smooth homology sphere of dimension $n$, whose fundamental group is a non-trivial perfect group. Let $\K$ be any triangulation of such homology sphere (which exists then by Whitehead's theorem~\cite{whitehead1940c1}). It is thus a simplicial complex satisfying the following three properties:
\begin{enumerate}[label=(\roman*)]
\item\label{p} $\K$ is path-connected.
\item\label{pi1} $\pi_1(\K)$ is a non-trivial perfect group.
\item\label{H} $\widetilde{H}_i(\K,\Z)=0$ for $i=0,\ldots,n-1$.
\end{enumerate}
 The simplicial complex $\susp(\K)$ is simply connected because of item~\ref{p} and inequality~\eqref{eq:conn-susp}. We can therefore apply the Hurewicz theorem to $\susp(\K)$. Since $\widetilde{H}_{i+1}(\susp(\K),\Z) = \widetilde{H}_i(\K,\Z)$ for all $i$ (by~\eqref{eq:homo-susp}), we conclude that $\susp(\K)$ is $n$-connected.
\end{proof}

Further, considering the importance of the connectivity of the suspension for the study of box complexes, we provide here a characterization of equality in~\eqref{eq:conn-susp}.

\begin{proposition}\label{prop:susp-perfect}
For any path-connected simplicial complex $\K$, we have 
\[
\conn(\susp(\K)) = \conn(\K) + 1 \quad \Longleftrightarrow \quad  \pi_1(\K) \text{ is trivial or not a perfect group} \, .
\]
\end{proposition}

(Actually, this proposition is true for any topological space with the structure of a CW-complex.) If $\K$ is not path-connected, then we always have the equality in~\eqref{eq:conn-susp} (as it can easily be seen by using~\eqref{eq:homo-susp} and the Hurewicz theorem).

\begin{proof}[Proof of Proposition~\ref{prop:susp-perfect}] By inequality~\eqref{eq:conn}, $\susp(\K)$ is simply connected because $\K$ is path-connected. Therefore, the Hurewicz theorem combined with the fact that 
$\widetilde H_{i+1}(\susp(\K),\Z) = \widetilde H_i(\K,\Z)$ for all $i\geq -1$ implies the following equality:
\begin{equation}\label{eq:conn-min}
\conn(\susp(\K)) = {\min}{\left\{n \colon \widetilde H_n(\K,\Z) \neq 0\right\}} \,.
\end{equation}

Suppose now that $\pi_1(\K)$ is trivial. Then by the Hurewicz theorem, the topological connectivity of $\K$ is the smallest $n$, minus $1$, such that $\widetilde H_n(\K,\Z)$ is non-zero. With the help of equality~\eqref{eq:conn-min}, we get the desired equality.

Suppose then that $\pi_1(\K)$ is not perfect. Then the topological connectivity of $\K$ is $0$. By the Hurewicz theorem, $\widetilde H_1(\K,\Z)$ is the abelianization of $\pi_1(\K)$. Since this latter is not a perfect group, $\widetilde H_1(\K,\Z)$ is non-zero. Since $\widetilde H_0(\K,\Z)$ is zero ($\K$ is path-connected), equality~\eqref{eq:conn-min} implies the desired equality.

Suppose now that $\pi_1(\K)$ is a non-trivial perfect group. Then again the topological connectivity of $\K$ is $0$. By the Hurewicz theorem, $\widetilde H_1(\K,\Z)$ vanishes, which implies according to equality~\eqref{eq:conn-min} that $\conn(\susp(\K)) > 1$.
\end{proof}

\subsection{Box complexes of join} Given two graphs $G$ and $H$, their {\em join} $G * H$ is the graph formed from disjoint copies of $G$ and $H$ and by connecting each vertex of $G$ to each vertex of $H$. One of the main new results of this paper, mentioned in the introduction, is Theorem~\ref{thm:join-B} below. It is a result relating the box complex $\B(G * H)$ of the join of two graphs $G$ and $H$ with that of the join of their box complexes $\B(G)$ and $\B(H)$. For $\B_0(\cdot)$, we have even a stronger result, with a simpler proof.

\begin{proposition}\label{prop:join-B0}
For every pair of graphs $G, H$, the complexes $\B_0(G * H)$ and $\B_0(G) * \B_0(H)$ are $\Z_2$-homeomorphic.
\end{proposition}

\begin{proof}
We prove a more general fact: $\sd(\B_0(G * H))$ and $\sd(\B_0(G)) *\sd(\B_0(H))$ are $\Z_2$-isomorphic. (Given a simplicial complex $\K$, we denote by $\sd(\K)$ its first barycentric subdivision.) 

Let $A'$ and $A''$ be two disjoint subsets of $V(G * H)$ such that $(G*H)[A',A'']$ is a complete bipartite subgraph of $G*H$. Write $A'$ as $A'_G \cup A'_H$ where $A'_G = A' \cap V(G)$ and $A'_H = A' \cap V(H)$, and do similarly for $A''$. For $A'$ and $A''$ not both empty, set
$f(A' \uplus A'') = (A'_G \uplus A''_G) \uplus (A'_H \uplus A''_H)$. It is a simplicial $\Z_2$-map from $\sd(\B_0(G * H))$ to $\sd(\B_0(G)) *\sd(\B_0(H))$. 

Let $B'$ and $B''$ (resp.\ $C'$ and $C''$) be two disjoint subsets of $V(G)$ (resp.\ $V(H)$) such that $G[B',B'']$ (resp.\ $H[C',C'']$) is a complete bipartite subgraph of $G$ (resp.\ $H$). For $B'$, $B''$, $C'$, and $C''$ not all empty, set
$g\big((B'\uplus B'') \uplus (C'\uplus C'')\big) = (B' \uplus C') \uplus (B'' \uplus C'')$. It is a simplicial $\Z_2$-map from $\sd(\B_0(G)) *\sd(\B_0(H))$ to $\sd(\B_0(G * H))$.

It is immediate that $f\circ g = g\circ f= \operatorname{id}$.
\end{proof}

A similar result is not true for the box complex $\B(G)$. (With Theorem~\ref{thm:dec_B0}, this might show that $\B_0(G)$ is a bit more handy than $\B(G)$.) Yet, we have the following weaker version.

\begin{theorem}\label{thm:join-B}
For every pair of graphs $G,H$, the complexes $\B(G * H)$ and ${\susp}{( \B(G) * \B(H))}$ are homotopy equivalent.
\end{theorem}

\begin{example}
    Consider the case where $G$ is the complete graph $K_m$ and $H$ the complete graph $K_n$. Then $G*H$ is $K_{m+n}$. 
    
    We have $\B_0(G)$, $\B_0(H)$, and $\B_0(G*H)$ respectively $\Z_2$-homeomorphic to $S^{m-1}$, $S^{n-1}$, and $S^{m+n-1}$ (Example~\ref{ex:box0-complete}). Thus $\B_0(G*H)$ and $\B_0(G)*\B_0(H)=S^{m-1}*S^{n-1}$ are indeed $\Z_2$-homeomorphic, as expected by Proposition~\ref{prop:join-B0}, since they are equal.

    The complexes $\B(G)$, $\B(H)$, and $\B(G*H)$ are homotopy equivalent respectively to $S^{m-2}$, $S^{n-2}$, and $S^{m+n-2}$ (Example~\ref{ex:box-complete}). Thus, $\B(G*H)$ and $\susp(\B(G)*\B(H))$ are indeed homotopy equivalent.
\end{example}

When $H$ is a single vertex, Theorem~\ref{thm:join-B} is a result by Csorba~\cite[Lemma 7.2]{csorba2007homotopy}, who actually proved the stronger $\Z_2$-homotopy equivalence for this special case. We conjecture that the $\Z_2$-homotopy equivalence holds actually for the general case as well.

Note that, by applying $\susp(\cdot)$ on each of $\B(G * H)$ and ${\susp}{( \B(G) * \B(H))}$ in Theorem~\ref{thm:join-B}, we get a weaker version of Proposition~\ref{prop:join-B0}, namely that $\B_0(G * H)$ and $\B_0(G) * \B_0(H)$ are homotopy equivalent. (Here, we see the suspension as the join with $S^0$, and use the commutativity and associativity of the join.)

The main tool of the proof of Theorem~\ref{thm:join-B} is the Fiber theorem of Quillen~\cite[Proposition 1.6]{quillen1978homotopy}, which states that an order preserving map $f \colon P\to Q$ of posets is a homotopy equivalence if $f^{-1}(Q_{\preceq q})$ is
contractible for every $q\in Q$. The notation $Q_{\preceq q}$ stands for the set $\{x\in Q\colon x\preceq q\}$. (An elementary proof of Quillen's Fiber theorem has been given by Barmak~\cite{barmak2011quillen}.)

\begin{proof}[Proof of Theorem~\ref{thm:join-B}]
We actually prove that $\N(G * H)$ and $\susp( \N(G) * \N(H))$ are homotopy equivalent. Since the join operation preserves homotopy equivalence (see, e.g., \cite[Exercise 2, Section 4.2]{matousek2008using}), the desired result will follow then from Theorem~\ref{thm:NB}.

Let $P$ be the face poset of $\N(G * H)$ and $Q$ that of $\N(G) * \N(H)$ (note that face posets do not contain empty faces). Define a new poset $\overline Q$ as $Q \cup \{r,s\}$, where $r$ and $s$ are two incomparable extra elements such that $q \prec r$ and $q \prec s$ for all $q \in Q$. Since the order complex of $\overline Q$ is homeomorphic to the suspension of the order complex of $Q$, the conclusion will follow from the proof that $P$ and $\overline Q$ are homotopy equivalent. This will be done by applying Quillen's Fiber theorem to the following map.

Let $\varphi\colon P \rightarrow \overline Q$ be defined for $A \subseteq V(G*H)$ such that $\CN_{G*H}(A)\neq\varnothing$ by
\[ \varphi(A) = \begin{cases} 
      A & \text{if $\CN_{G}(A\cap V(G))\neq\varnothing$ and $\CN_{H}(A\cap V(H))\neq\varnothing$,}\\
      r & \text{if $\CN_{G}(A\cap V(G))=\varnothing$,}\\ 
       s & \text{if $\CN_{H}(A\cap V(H))=\varnothing$.}
   \end{cases}
\]
We check that $\varphi$ is well defined, namely that 
it is impossible that both $\CN_{G}(A\cap V(G))$ and $\CN_{H}(A\cap V(H))$ are simultaneously empty for $A \subseteq V(G*H)$ such that $\CN_{G*H}(A)\neq\varnothing$. This is clear when $A\cap V(G)$ or $A\cap V(H)$ are empty. (Recall that the common neighbors of the empty set is the full vertex set.) So, suppose that both $A\cap V(G)$ and $A\cap V(H)$ are nonempty. Pick $v$ in $\CN_{G*H}(A)$. Without loss of generality, $v$ is in $V(G)$, and thus $\CN_G(A\cap V(G))$ is nonempty.

The map $\varphi$ is obviously order preserving. In order to use Quillen's Fiber theorem, it remains to check that $D_q \coloneqq {\varphi}^{-1}(\overline Q_{\preceq q})$ is contractible for every $q \in \overline Q$. 

First, consider the case when $q \in Q$. In this case, $q$ is actually a subset $A$ of $V(G * H)$ such that $\CN_{G}(A\cap V(G))\neq\varnothing$ and $\CN_{H}(A\cap V(H))\neq\varnothing$. Since $\varphi(A)=A$, we have $A \in D_A$. Moreover, by definition of $D_A$, any other $B\in D_A$ is such that $\varphi(B) \preceq A$ and thus such that $\varphi(B) = B$, which implies $B \preceq A$. The set $A$ is the unique maximal element of $D_A$, which shows that $D_A$ is a cone and therefore contractible. 

Consider now the case when $q$ is $r$ or $s$. Without loss of generality, we assume that $q$ is $r$. We have already checked that it is impossible that both $\CN_{G}(A\cap V(G))$ and $\CN_{H}(A\cap V(H))$ are simultaneously empty for $A \subseteq V(G*H)$ such that $\CN_{G*H}(A)\neq\varnothing$. It means that if $A$ is in $D_r$, then $\CN_{H}(A\cap V(H))$ is nonempty, and thus $\varphi(A \cup V(G)) = A \cup V(G)$, which implies in particular that $A \cup V(G)$ belongs to $D_r$. In other words, the map $\psi\colon A \mapsto A \cup V(G)$ is an order preserving map $D_r \rightarrow D_r$. Note that $\psi(A) \succeq A$ for every $A$ is $D_r$. In other words, we have $\psi \succeq \operatorname{id}_{D_r}$ for the induced order on the self-maps of $D_r$. A standard result in combinatorial topology says then that the image of $\psi$ is homotopy equivalent to $D_r$ (see, e.g.,~\cite[Proposition 12]{vzivaljevic2005wi}). Since $V(G)$ belongs to $D_r$, it also belongs to the image of $\psi$, and it is its unique minimal element. Thus the image of $\psi$ is contractible and so is $D_r$.

We can therefore apply Quillen's Fiber theorem and we get that $\varphi$ is a homotopy equivalence, which finishes the proof.
\end{proof}

\section{Commenting Figure~\ref{fig}}\label{sec:comment}

\subsection{Arrows}\label{subsec:arrows} In this section, we review every arrow of Figure~\ref{fig}: for each arrow, we provide for the corresponding inequality a proof or a reference to the literature, and discuss the possible gaps.

\begin{proof}[Inequality \eqref{ar1}]
We prove the inequality. Any clique $K_n$ of $G$ induces a homomorphism from $K_n$ to $G$. This homomorphism becomes a simplicial $\Z_2$-map from $\B(K_n)$ to $\B(G)$ when it is read at the level of box complexes. In Example~\ref{ex:box-complete}, we have seen that $\B(K_n)$ is homotopy equivalent to $S^{n-2}$; the same argument actually shows its $\Z_2$-homotopy equivalence with $S^{n-2}$. Therefore $\coind(\B(G)) \geq n-2$.

The gap can be arbitrarily large: the clique number of the Kneser graph $\KG(3k-1,k)$ is $2$ while we have $\conn(\N(\KG(3k-1,k)))+3 = k+1$, as stated at the beginning of Section~\ref{subsec:spec-graph}; Theorem~\ref{thm:NB} shows then that $\conn(\B(\KG(3k-1,k)))+3 = k+1$ as well, which is a lower bound on ${\coind}{(\B(\KG(3k-1,k)))}+2$, established by inequality~\eqref{ar2}. (We have actually ${\coind}{(\B(\KG(3k-1,k)))}+2 = k+1$ by using the upper bound given by the chromatic number.)
\end{proof}

\begin{proof}[Inequality \eqref{ar2}]
The inequality is a special case of inequality~\eqref{a1} of Section~\ref{subsec:gen-rel-par}.

The gap can be arbitrarily large: consider for instance the graph formed by two disjoint copies of $K_{n+2}$; its box complex $\B(G)$ is $\Z_2$-homotopy equivalent to the disjoint union of two $n$-dimensional spheres; it is thus not $0$-connected, while it has a coindex equal to $n$.
\end{proof}

\begin{proof}[Inequality \eqref{ar3}]
The inequality is a direct consequence of Theorem~\ref{thm:susp} combined with inequality~\eqref{eq:conn-susp}.

We prove that the gap can be arbitrarily large. This fact, which has not been emphasized in the literature yet, is another evidence on the better behavior of $\B_0(G)$ for providing efficient lower bounds on the chromatic number.

Fix an integer $n > 5$. Choose any simplicial complex $\K\coloneqq \K_n$ as in
Proposition~\ref{prop:susp}. Consider now $\K\times S^{n+1}$ and equip this space with the $\Z_2$-action that acts on the first component trivially and on the second component as the antipodal action. By Theorem~\ref{thm:csorba}, there is a graph $G$ such that $\B(G)$ and $ \K\times S^{n+1}$ have the same $\Z_2$-homotopy type. Now, by Theorem~\ref{thm:susp}, $\B_0(G)$ and $\susp(\K \times S^{n+1})$ are $ \Z_2 $-homotopy equivalent. So, to finish the proof it is enough to show that the following two equalities hold: 
\begin{equation}\label{eq:KSn}
\conn( \K \times S^{n+1} )=0 \quad \mbox{and} \quad \conn(  \susp(\K \times S^{n+1}) ) \geq n \, .
\end{equation}
The first equality of \eqref{eq:KSn} is a direct consequence of $\K\times S^{n+1} $ being path-connected and $\pi_1( \K\times S^{n+1} )=\pi_1(\K)\times \pi_1( S^{n+1} ) = \pi_1(\K) \neq 0$. To show the second equality, note first that $ \widetilde{H}_i( \K\times S^{n+1}, \Z) = \widetilde {H}_i(\K, \Z)$ for all $i < n$: indeed, the $i$th homology of a CW-complex $X$ depends only on its $(i+1)$-skeleton; considering $S^{n+1}$ as a CW-complex with just two cells (one zero-cell and one $(n+1)$-cell), the $(i+1)$-skeleton of $\K\times S^{n+1}$ is just the $(i+1)$-skeleton of $\K$ when $i < n$. We have $ \widetilde {H}_i(\K, \Z)=0$ for all $i < n$ because $\K$ was chosen like this in the proof of Proposition~\ref{prop:susp} (and anyway it is a consequence of the topological connectivity of $\susp(\K)$ being exactly $n$). As $\widetilde{H}_{i+1}( \susp(\K\times S^{n+1}), \Z  )= \widetilde {H}_i( \K\times S^{n+1}, \Z)$ for all $i\geq 0$, the Hurewicz theorem implies the second equality of \eqref{eq:KSn} ($\susp(\K\times S^{n+1}) $ is simply connected). 
\end{proof}

\begin{proof}[Inequality \eqref{ar4}]
Any $\Z_2$-map between two spaces can be lifted to a $\Z_2$-map between their suspensions. This implies the desired inequality because the suspension of a sphere is a sphere of one dimension higher and because the suspension of $\B(G)$ is $\Z_2$-homotopy equivalent to $\B_0(G)$ (Theorem~\ref{thm:susp}).

Simonyi, Tardos, and~Vre\'cica~\cite{simonyi2009local} have shown that the coindex of any $2$-dimensional compact manifold with even genus is equal to $1$, while the coindex of its suspension is equal to $3$. Theorem~\ref{thm:csorba} implies then that there are graphs for which the gap between $\coind(\B(G))$ and $\coind(\B_0(G))$ can be equal to $2$. In Section~\ref{subsec:Briesk} below, we show the existence of a topological $\Z_2$-space with a gap of $3$ between its coindex and the coindex of its suspension. Since this $\Z_2$-space is actually smooth and compact, it admits a triangulation that is $\Z_2$-equivariant~\cite{illman1978smooth}. There exists therefore a simplicial $\Z_2$-complex with a gap of $3$ between its coindex and the coindex of its suspension, and Theorem~\ref{thm:csorba} implies then that there are graphs for which the gap between $\coind(\B(G))$ and $\coind(\B_0(G))$ can be equal to $3$. Whether this gap can be larger is unknown.
\end{proof}

\begin{proof}[Inequality \eqref{ar5}]
The inequality is a special case of inequality~\eqref{a1} of Section~\ref{subsec:gen-rel-par}.

We prove that the gap can be arbitrarily large. We take the same example as in the proof of inequality~\eqref{ar2}, namely two disjoint copies of a complete graph, which we choose here to be $K_{n+1}$. Denote this graph by $G$. Its box complex $\B_0(G)$ can be described as follows. Let $\partial \Diamond_1^{n+1}$ and $\partial \Diamond_2^{n+1}$ be two disjoint copies of the boundary of the $(n+1)$-dimensional cross-polytope. Pick in each of them two opposite facets: $F_1$ and $F_1'$ for $\partial \Diamond_1^{n+1}$ and $F_2$ and $F_2'$ for $\partial \Diamond_2^{n+1}$. Then $\B_0(G)$ is isomorphic to $\partial \Diamond_1^{n+1} \cup \partial \Diamond_2^{n+1} \cup (F_1 * F_2) \cup (F'_1 * F_2')$. The boundary of an $(n+1)$-dimensional cross-polytope being an $n$-dimensional sphere, the coindex of $\B_0(G)$ is at least $n$. On the other hand, contracting each $F_1*F_2$ and $F_1'*F'_2$ to a point shows that $\B_0(G)$ is homotopy equivalent to two $n$-spheres sharing their North and South poles. It is a standard exercise to check that such space is homotopy equivalent to the wedge of $S^1$ and two disjoint copies of $S^n$.
\end{proof}

\begin{proof}[Inequality \eqref{ar6}]
The equality between $\conn_{\Z_2}(\B(G))+1$ and $\conn_{\Z_2}(\B_0(G))$ is a consequence of equation~\eqref{eq:connZ2} and Theorem~\ref{thm:susp}. The inequality is a special case of inequality~\eqref{a2} of Section~\ref{subsec:gen-rel-par}.

We prove that the gap can be arbitrarily large. Take any simplicial complex $\K$ with $\widetilde H_1(\K,\Z)=\Z_3$ and $\widetilde H_i(\K,\Z)=0$ for $i\neq 1$ (such a simplicial complex exists; see, e.g., \cite[Example 2.40]{hatcher2005algebraic}). Fix a natural number $n\geq 3$ and set 
$X=\K\times S^{n+1}$. Equip this space with the $\Z_2$-action that acts on the first component trivially and on the second component as the antipodal action. The homology groups of $\susp(X)$ over $\Z$ vanishes in every dimension less than $n+1$, except in dimension $2$, in which it is $\Z_3$ (see equation~\eqref{eq:homo-susp}). So, by Hurewicz's theorem, its topological connectivity is at most one. On the other hand, $\widetilde H_i(\susp(X),\Z_2)=0$ for $i\leq n$ by the universal coefficient theorem, which implies $\conn_{\Z_2}(\susp(X)) \geq n$. Now, using Theorem~\ref{thm:csorba}, we are done.
\end{proof}

\begin{proof}[Inequality \eqref{ar7}]
The inequality was proved by Alishahi and Hajiabolhassan~\cite[Theorem 1]{Alishahi2018AGO}. Actually, they proved a stronger result, namely that the {\em altermatic number}---a parameter they introduced, which dominates the colorability defect---is a lower bound on $\coind(\B_0(G))+1$.

The gap can be arbitrarily large as shown by considering the case where $G=K_n$ is the complete graph and its Kneser representation $\HH$ is a graph formed by a perfect matching: in that case $\cd(\HH)=0$ and $\coind(\B_0(G))=n-1$. Yet, choosing $\HH$ as the hypergraph with $n$ vertices and $n$ edges formed by all possible singletons leads to a tight bound. 

Nevertheless, there are graphs for which the gap is arbitrarily large for every Kneser representation. An example is obtained by taking the join of an arbitrary number of $C_5$. Consider the graph $G=C_5^{*n}$.

On the one hand, the colorability defect of any Kneser representation of $C_5$ is at most one. Indeed, in such a Kneser representation, every pair of non-adjacent vertices share a common element (as edges of the Kneser representation) that is not shared by any other vertex; interpreting every such common element as an edge of a new graph between the non-adjacent vertices shows that every Kneser representation contains a $C_5$; removing any edge of this $C_5$ leads to a path, whose edges can be two-colored. Any Kneser representation of $G$ is the disjoint union of $n$ Kneser representations of $C_5$ and thus has colorability defect upper bounded by $n$. 

On the other hand, $\B(C_5)$ is homotopy equivalent to $S^1$ (Example~\ref{ex:box-C2n+1}). Theorem~\ref{thm:join-B} shows then that $\B(G)$ is homotopy equivalent to $S^{3n-2}$. Since the chromatic number of $G$ is $3n$, it makes the bound $\coind(\B(G))+2$ tight, which in turn implies that $\coind(\B_0(G))+1=3n$.
\end{proof}

\begin{proof}[Inequality \eqref{ar8}]
The equality between $\hind(\B(G))+1$ and $\hind(\B_0(G))$ is a consequence of equation~\eqref{eq:hind} and Theorem~\ref{thm:susp}. The inequality is a special case of inequality~\eqref{b1} of Section~\ref{subsec:gen-rel-par}.

The authors believe that the gap can be positive (and even arbitrarily large), but they have not been able to come up with a concrete example.
\end{proof}

\begin{proof}[Inequality \eqref{ar9}]
The inequality is a special case of inequality~\eqref{b2} of Section~\ref{subsec:gen-rel-par}.

We prove that the gap can be arbitrarily large. Consider the space $X$ formed by $S^n$ and two copies of $S^1$: one $S^1$ is attached to the North pole, and one $S^1$ is attached to the South pole. We assume that the $\Z_2$-action is the central symmetry on $S^n$ and exchanges the two $S^1$'s. We have $\hind(X) \geq n$ because $\coind(X) = n$ and $\conn_{\Z_2} (X) = 0$. We finish with Theorem~\ref{thm:csorba}.
\end{proof}

\begin{proof}[Inequality \eqref{ar10}]
The inequality is a special case of inequality~\eqref{c} of Section~\ref{subsec:gen-rel-par}. The gap can be arbitrarily large. Indeed, the odd projective space $\R P^{2n+1}$ can be equipped with a free $\Z_2$-action so that its index is at least $n$~\cite{stolz1989level}, while its cohomological index is $1$~\cite[Theorem 6.6]{conner1960}, and we finish with Theorem~\ref{thm:csorba}.
\end{proof}

\begin{proof}[Inequality \eqref{ar11}]
We always have $\ind(X) \leq \ind(\susp(X)) \leq \ind(X) + 1$ for any topological $\Z_2$-space $X$. The left-hand inequality is immediate from the definition of the index and the right-hand inequality is a consequence of the right inequality in equation~\eqref{eq:coind} (which is by the way also immediate). This explains why inequality~\eqref{ar11} holds, and also shows that the gap can be at most one. There are examples achieving this gap~\cite{csorba2007homotopy}. We finish with Theorem~\ref{thm:csorba}.
\end{proof}

\begin{proof}[Inequality \eqref{ar12}]
The inequality was proved by Csorba et al.~\cite{csorba2004box}. The gap can be arbitrarily large as shown by arbitrary complete bipartite graphs; see Example~\ref{ex:box-bipcomplete}.
\end{proof}

\begin{proof}[Inequality \eqref{ar13}]
The inequality was proved by Simonyi et al.~\cite{simonyi2013colourful}. (Note that it is not a special case of inequality~\eqref{d} of Section~\ref{subsec:gen-rel-par}, because there is an extra barycentric subdivision.) The question on how large the gap can be is completely open.
\end{proof}

\begin{proof}[Inequality \eqref{ar14}]
The inequality was proved by Simonyi et al.~\cite{simonyi2013colourful}. The question on how large the gap can be is completely open.
\end{proof}

\begin{proof}[Inequality \eqref{ar15}]
The inequality is obvious from the definition. The gap can be arbitrarily large as shown by graphs with no $C_4$ ($=K_{2,2}$) and arbitrarily large chromatic number, which exist by Erd\H{o}s's result~\cite{erdos1959graph}.
\end{proof}

For all inequalities there are graphs realizing them as equalities. Indeed, for each of the lower bounds $\omega(G)$, $\cd(\HH)$, and $\conn(\B(G))+3$, there are many graphs for which they are tight (obtained from the collection of perfect graphs or Kneser graphs). Actually, the complete graph is an example for which the three bounds are tight simultaneously (with $\HH$ being the complete $1$-uniform hypergraph).

\subsection{A topological \texorpdfstring{$\Z_2$}{Z2}-space with a gap of three between its coindex and the coindex of its suspension}\label{subsec:Briesk}

The \emph{Brieskorn space} $M_{p,q,r}$ is defined as the intersection of the complex algebraic surface $z_1^p + z_2^q + z_3^r = 0$ with the unit $5$-dimensional sphere $|z_1|^2 + |z_2|^2 + |z_3|^2 = 1$. Here, $p$, $q$, and $r$ should be integer numbers non-smaller than $2$. These spaces are well-studied objects from algebraic geometry. When $p$, $q$, and $r$ are odd, they are naturally equipped with the free $\Z_2$-action inherited from the antipodal action on the unit sphere.

Existence of spaces with the following property was implicitly asked in Section 5 of the already cited paper by Simonyi, Tardos, and Vrecica~\cite{simonyi2009local} to study the gap in inequality~\eqref{ar4}.

\begin{proposition}\label{prop:Briesk}
We have $\coind(\susp(M_{p,q,r})) = 4$ and $\coind(M_{p,q,r})=1$, whenever $p$, $q$, and $r$ are pairwise coprime odd integers larger than $2$. 
\end{proposition}

We need a preliminary property about Brieskorn spaces $M_{p,q,r}$. This property is actually common knowledge in algebraic topology but we were not able to find any bibliographical reference. For sake of reader's understanding, we write here a complete proof.

\begin{lemma}\label{lem:Briesk}
The second and third homotopy groups of $M_{p,q,r}$ vanish when $1/p + 1/q + 1/r - 1$ is negative.
\end{lemma}

\begin{proof}
In this case, $M_{p,q,r}$ is diffeomorphic to a coset space of the form $\mathsf{G}/\mathsf{\Pi}$ where $\mathsf{G}$ is the universal cover of $\operatorname{SL}(2,\mathbb R )$ and $\mathsf{\Pi}$ is a discrete subgroup of $\mathsf{G}$~\cite[Section 6]{milnor1975}. As a topological space, $\operatorname{SL}(2,\mathbb R)$ is homeomorphic to $S^1\times\mathbb{R}^2$. (This can be derived quite easily from the Iwasawa decomposition of $\operatorname{SL}(2,\mathbb R)$; see \cite{iwasawa}.) Since the universal cover of $S^1$ is $\mathbb{R}$, the space $\mathsf{G}$ is $\mathbb{R}^3$ and thus contractible. The subgroup $\mathsf{\Pi}$ being discrete, it acts in a properly discontinuous way on $\mathsf{G}$, which implies that the quotient map $\mathsf{G} \rightarrow \mathsf{G}/\mathsf{\Pi}$ is a covering map (see \cite[Proposition 1.40]{hatcher2005algebraic} where the author prefers to use ``covering space action'' instead of ``properly discontinuous action''). This implies in turn that there is a isomorphism between the homotopy groups $\pi_n(\mathsf{G})$ and $\pi_n(\mathsf{G}/\mathsf{\Pi})$ for $n\geq 2$~\cite[Proposition 4.1]{hatcher2005algebraic}. In particular, $\mathsf{G}$ being contractible, the second and third homotopy groups of $M_{p,q,r}$ vanish.
\end{proof}

\begin{proof}[Proof of Proposition~\ref{prop:Briesk}]
To ease the notation, denote by $X$ any Brieskorn space $M_{p,q,r}$ with $p$, $q$, and $r$ pairwise coprime integers larger than $2$. Then $X$ has the following properties:
\begin{enumerate}[label=(\alph*)]
    \item\label{sphere} it is an integral $3$-dimensional homology sphere (see the ``Historical Remarks'' in the introduction of the paper by Milnor~\cite{milnor1975}).
    \item\label{homotop} the homotopy groups $\pi_2(X)$ and $\pi_3(X)$ are $0$ (Lemma~\ref{lem:Briesk}).
\end{enumerate}

We claim that $\coind(X)=1$ and $\conn(\susp(X))\geq 3$, which implies that $\coind(\susp(X) = 4$ (by inequalities~\eqref{a1} and~\eqref{e} from Section~\ref{subsec:gen-rel-par}). 
We finish the proof by showing these two equalities.

 By property~\ref{sphere}, the reduced $0$-dimensional homology of $X$ vanishes, which makes it path-connected. By inequality~\eqref{a1}, we have then $\coind(X) \geq 1$. Assume for a contradiction that $\coind(X) \geq 2$. This implies the existence of a continuous $\Z_2$-map $S^2 \rightarrow X$. Because of property~\ref{homotop}, we can extend this map to a continuous $\Z_2$-map $S^4 \rightarrow X$ (connectivity is used to get a $B^3\rightarrow X$ map that is a $\Z_2$-map on the boundary, which allows then to get a $\Z_2$-map $S^3\rightarrow X$, and then we repeat again this construction for the next dimension), which is not possible because $X$ is $3$-dimensional (we use again inequality~\eqref{e}).

The space $\susp(X)$ is simply connected because $X$ is path-connected (by inequality~\eqref{eq:conn-susp}). Hence, Hurewicz's theorem applies and equality~\eqref{eq:homo-susp} together with property~\ref{sphere} imply that $\conn(\susp(X))\geq 3$. 
\end{proof}

\begin{remark}\label{rem:coind-susp}
Proposition~\ref{prop:Briesk} ensures that there exists a topological $\Z_2$-space with a gap of $3$ between the coindex of its suspension and its coindex. To achieve an arbitrary gap of $k$, the same proof works provided there exists a triangulable topological $\Z_2$-space with
\begin{itemize}
    \item vanishing reduced homology up to dimension $k-1$.
    \item vanishing homotopy groups from dimension $2$ to $k$.
\end{itemize}
We believe that such a space should exist for arbitrarily large $k$.
\end{remark}

\subsection{Absent arrows}

Let two graph parameters be {\em comparable} if for all graphs they are ordered in the same way. Some parameters are not compared in Figure~\ref{fig}. They come in two categories: some pairs of such parameters are really not comparable; for other pairs, the question whether they are comparable or not is still open. The next proposition shows that Figure~\ref{fig} is exhaustive for $\omega(G)$: if there is no path in the figure from $\omega(G)$ to some parameter, this latter is not comparable with $\omega(G)$.

\begin{proposition}\label{prop:uncomp}
None of the parameters $\max_{\HH}\cd(\HH)$ (over all Kneser representations), $\conn(\B(G))+3$, $\conn(\B_0(G))+2$, and $\conn_{\Z_2}(\B(G))+3)$ is comparable with $\omega(G)$. Moreover, each of them can be 
arbitrarily larger than $\omega(G)$ and also arbitrarily smaller (the difference can be arbitrarily large in one direction or the other).
\end{proposition}

\begin{proof} For each parameter, we give a first family that shows that it can be made arbitrarily larger, and then we give a second family that shows that it can be made arbitrarily smaller.

$\max_{\HH}\cd(\HH)$: for Kneser graphs $\KG(3k,k+1)$ (see Section~\ref{subsec:spec-graph}), this bound is tight (equal to $k$) while their clique number is $2$; in the proof of inequality~\eqref{ar7}, it has been shown that $C_5^{*n}$ has $\max_{\HH}\cd(\HH))$ at most $n$ for every Kneser representation, while its clique number is equal to $2n$.

$\conn(\B(G))+3$, $\conn(\B_0(G))+2$, and $\conn_{\Z_2}(\B(G))+3$: again, for Kneser graphs $\KG(3k,k+1)$, these bounds are tight (equal to $k$, see Section~\ref{subsec:spec-graph}) while their clique number is $2$; however, disjoint union of two complete graphs gives a bound of $3$ while the clique number is arbitrarily high.
\end{proof}

One of the messages of Proposition~\ref{prop:uncomp} is that none of $\omega(G)$ and $\conn(\B(G))+3$ provides a better bound on the chromatic number than the other. This must however be mitigated: it is known that ``generically'' the clique number provides a better bound. This is for instance formalized by a result due to Kahle~\cite{kahle2007neighborhood}: 
while the clique number of the Erd\H{o}s--R\'enyi graph $G(n,1/2)$ is almost surely asymptotically equivalent to $2\log_2(n)$, the value of $\conn(\B(G(n,1/2)))/\log_2(n)$ lies almost surely asymptotically between $1$ and $4/3$. (The result is stated for the neighborhood complex, which is equivalent because of Theorem~\ref{thm:NB}.) It is worth noticing that the bound $b(G)$ is used in that paper to get the probabilistic upper bound on the connectivity of the neighborhood complex. Anyway, the bound provided by $\omega(G)$ is very weak for $G(n,1/2)$ because, by a celebrated result by Bollob\'as~\cite{bollobas1988chromatic}, the chromatic number of this latter is almost surely asymptotically equivalent to $n/(2\log_2(n))$.

Regarding parameters $b(G)$ and $\chi(G)$: complete bipartite graphs show that $b(G)$ can be arbitrarily larger than $\chi(G)$; and $C_4$-free graphs with high chromatic number (see proof of inequality~\eqref{ar15} where such graphs have already been used) show that $\chi(G)$ can be arbitrarily larger than $b(G)$. Unfortunately, for the other pairs that are not compared, we were only able to obtain partial answers, like the following ones.

\begin{proposition}\label{prop:conn-coind}
Neither $\conn(\B_0(G))$, nor $\conn_{\Z_2}(\B_0(G))$ is comparable with $\coind(\B(G))$. Moreover, each of them can be 
arbitrarily smaller than $\coind(\B(G))$.
\end{proposition}

\begin{proof}
We are going to see that there exists a (triangulable) topological $\Z_2$-space $X$ such that $\coind(X)=1$ and $\conn(\susp(X)) = 3$, and a (triangulable) topological $\Z_2$-space $Y$ such that $\coind(Y)= n$ and $\conn_{\Z_2} (\susp(Y)) = 1$. As usual, this provides then the desired conclusion with Theorem~\ref{thm:csorba}.

Such a space $X$ was actually described in the proof of Proposition~\ref{prop:Briesk} since we established there that, for some values of $p,q,r$, the coindex of the Brieskorn space $M_{p,q,r}$ is $1$, while the connectivity of its suspension is $3$ (see proof of Proposition~\ref{prop:Briesk}).

The space $Y$ is formed by $S^n$ and two copies of $S^1$: one $S^1$ is attached to the North pole, and one $S^1$ is attached to the South pole. We assume that the $\Z_2$-action is the central symmetry on $S^n$ and exchanges the two $S^1$'s. We have $\coind(Y)= n$ and $\conn_{\Z_2} (\susp(Y)) = 1$.
\end{proof}

It is plausible that each of $\conn(\B_0(G))$ and $\conn_{\Z_2}(\B_0(G))$ can also be arbitrarily larger than $\coind(\B(G))$; this would actually be a consequence of the existence of complexes as described in Remark~\ref{rem:coind-susp}. Proposition~\ref{prop:conn-coind} shows that $\conn_{\Z_2}(\B_0(G))+2$ can be smaller than $\coind(\B_0(G))+1$. Whether it can be larger is open. \label{page:coind-conn2}

\begin{question}\label{q:conncoindB0}
     Do there exist graphs for which $\conn_{\Z_2}(\B_0(G))+2$ is larger than $\coind(\B_0(G))+1$?
\end{question}

There are also several questions about the location of $\cd(\HH)$ with respect to the other topological bounds below the Borsuk--Ulam boundary. The most intriguing is probably the relation of $\cd(\HH)$ and $\conn(\B(G))+3$. The following result provides a partial answer.

\begin{proposition}
    The bound $\max_{\HH}\cd(\HH)$ can be arbitrarily smaller than $\conn(\B(G))+3$.
\end{proposition}

\begin{proof}
    For $G=C_5^{*n}$, we have seen in the proof of inequality~\eqref{ar7} that $\max_{\HH}\cd(C_5^{*n})$ is equal to $n$ and that $\B(G)$ is homotopy equivalent to $S^{3n-2}$, which implies that $\conn(\B(G))+3$ is equal to $3n$. 
\end{proof}

However, we have the following open question.\label{page:cdB}

\begin{question}\label{q:cdB}
    Do there exist graphs for which $\cd(\HH)$ is larger than $\conn(\B(G))+3$?
\end{question}

Note that the similar questions with $\conn(\B_0(G))+2$ and $\conn_{\Z_2}(\B_0(G))+2$ in place of $\conn(\B(G))+3$ are also open.

\section{Complementary remarks}\label{sec:compl}

\subsection{Further topological bounds}

The Hom complex $\Hom(K_2,G)$ can be alternatively defined as the poset of ``multihomomorphisms'' from $K_2$ to $G$, and replacing $K_2$ by another arbitrary graph $T$ provides another Hom complex $\Hom(T,G)$; see, e.g.,~\cite[Chapter 5, Section 5.9]{matousek2008using} or~\cite[Chapter 2, Section 2.3]{de2012course} for details. (More insight on the structure of Hom complexes is given by Dochtermann~\cite{Dochtermann09}.) Bj\"orner and Lov\'asz conjectured that $\conn(\Hom(T,G))+\chi(T)+1$ is a lower bound on $\chi(G)$ (remember that $\B(G)$ is homotopy equivalent to the order complex of $\Hom(K_2,G)$---Theorem~\ref{thm:BH}---and thus the conjecture generalizes the bound based on the connectivity of $\B(G)$).

This conjecture generated a lot of activity. Hoory and Linial~\cite{hoory2005counterexample} showed that the conjecture cannot be true for all graphs $T$, and Babson and Kozlov~\cite{babson2007proof} established it when $T$ is an odd cycle or a complete graph. This latter result has been considered as a breakthrough in topological combinatorics. Simplifications of the proof and complementary results were found soon after by {\v Z}ivaljevi\'c~\cite{zivaljevic2005parallel}, Schultz~\cite{schultz2009graph}, and Kozlov~\cite{kozlov2006cobounding}. Anyway, Schultz proved that we always have $\hind(\B(G))+2 \geq \hind(\Hom(C_{2r+1},G))+3$ for any $r$, which shows that box complexes are not yet outdated by more general Hom complexes.

Hom complexes have been recently extended to directed graphs by Dochtermann and Singh~\cite{dochtermann2023homomorphism}. They show how methods from equivariant topology can be used to provide obstructions to digraph homomorphism, for example by considering Hom complexes from a directed 3-cycle (the classical case of $K_2$ is recovered as a directed $2$-cycle). They leave it as an open problem whether these results could be fruitfully applied to naturally arising questions in directed graph theory.

\subsection{Open questions}\label{subsec:open} We collect here the main open questions met in the survey.

\subsubsection*{Improving Csoba's construction} Is it possible to improve Csorba's construction, which shows that every free $\Z_2$-simplicial complex is $\Z_2$-homotopy equivalent to a certain $\B(G)$? See Section~\ref{subsec:univers}, p.\pageref{page:csorba}.

\subsubsection*{Box complex of $s$-stable Kneser graphs} Not much is known about the topological nature of the box complex of $s$-stable Kneser graphs. Generalizing Theorem~\ref{thm:schultz} is a challenging and natural question, whose answer could be beneficial to the conjecture on the chromatic number of $s$-stable Kneser graphs (Section~\ref{subsec:spec-graph}, p.\pageref{page:sstable}). 

\subsubsection*{Decidability of topological parameters} Which of the parameters $\coind(\B(G))$, $\coind(\B_0(G))$, $\ind(\B(G))$, and $\ind(\B_0(G))$ are decidable? Which are not?
See Section~\ref{subsec:dec-compl}, p.\pageref{page:dec}.

\subsubsection*{Computational complexity of parameters}
 What is the complexity status of the computation of the parameters of Figure~\ref{fig}? Apart for $\chi(G)$, $b(G)$, $\cd(\HH)$, and $\omega(G),$ which are \NP-hard, the complexity status is open. The combinatorial parameters $\zig(G)$ and $\Xind(\Hom(K_2,G))$ might be appealing challenges. See Section~\ref{subsec:dec-compl}, p.\pageref{page:compl}.

\subsubsection*{Computational complexity and the Borsuk--Ulam theorem}
 Regarding complexity questions, maybe more fundamental than those mentioned above is Question~\ref{q:kneser} (Section~\ref{subsec:BU}, p.\pageref{q:kneser}), asking whether proving that Kneser graphs have chromatic number $n-2k+2$ is as hard as establishing the Borsuk--Ulam theorem.

\subsubsection*{Hedetniemi's conjecture beyond the chromatic number} For some parameters of Figure~\ref{fig}, a Hedetniemi-type relation is known to hold. But it is still open for several of them. This is Question~\ref{q:hedet} of Section~\ref{sec:product}, p.\pageref{page:hedet-paramG}. 

Moreover, similar questions can be asked for topological parameters attached to topological spaces, without any reference to graphs. See Section~\ref{sec:product}, p.\pageref{page:hedet-param-top}.

\subsubsection*{Gap between $\coind(\B(G))$ and $\coind(\B_0(G))$ and generalized Brieskorn spaces}
Brieskorn spaces offer examples of free $\Z_2$-spaces with coindex equal to $1$ and with the coindex of their suspension equal to $4$ (Proposition~\ref{prop:Briesk}). They are used to show that the gap in inequality~\eqref{ar4} can be at least $2$. Generalizations of Brieskorn spaces, as suggested in Remark~\ref{rem:coind-susp}, p.\pageref{rem:coind-susp}, would show the existence of a space with coindex $1$ and arbitrarily large connectivity of its suspension, and would show in turn that the gap in inequality~\eqref{ar4} can be arbitrarily large (a question that is still open).

\subsubsection*{Coindex and connectivity mod 2 for suspensions} It is not known whether $\conn_{\Z_2}(\B_0(G))+1$ can be larger than $\coind(\B_0(G))$; see Question~\ref{q:conncoindB0}, p.\pageref{page:coind-conn2}. This question can be formulated in a pure topological way: is there any topological $\Z_2$-space $X$ for which $\conn_{\Z_2}(\susp(X))+1$ is larger than $\coind(\susp(X))$?

\subsubsection*{Colorability defect and coindex} Whether the colorability defect lower bound can be sometimes better than $\coind(\B(G))+2$ is open; see Question~\ref{q:cdB}, p.\pageref{page:cdB}. Note that a positive answer would also contribute to the question about the gap between $\coind(\B(G))$ and $\coind(\B_0(G))$.

\bibliographystyle{amsplain}
\bibliography{box-complex-survey}

\appendix

\section{Proof of Theorem~\ref{thm:complex-coind}}

We first establish two technical lemmas. We denote the $d$-dimensional cross-polytope by $\Diamond^d$ (notation already met in Section~\ref{subsec:BU} when we proved the Borsuk--Ulam theorem from the inequality $\conn(\B(G))+3\leq\chi(G)$).

\begin{lemma}\label{lem:bk}
There exists a (polynomially computable) simplicial $\Z_2$-map $\sd(\B(K_{d+1})) \rightarrow \partial \Diamond^d$ (encoded as a Boolean $\Z_2$-circuit). 
\end{lemma}

\begin{proof}
    Identify the vertices of $K_{d+1}$ with the integers in $\{1,2,\ldots,d+1\}$. This way, the vertices of $\B(K_{d+1})$ are identified with the integers in $\{\pm 1,\pm 2,\ldots,\pm (d+1)\}$. For a vertex $v$ of $\B(K_{d+1})$, set $\mu(v) \coloneqq (-1)^vv$.

    Consider now a simplex $\sigma$ of $\B(K_{d+1})$ and order its vertices in the increasing order of their absolute values: $v_0,v_1,\ldots,v_{\dim\sigma}$. Since no simplex of $\B(K_{d+1})$ can have vertices of opposite values, there is no ambiguity in the definition of this ordering. Let $k$ be the number of sign changes in the sequence $\mu(v_0),\mu(v_1),\ldots,\mu(v_{\dim\sigma})$, plus one. Let $s$ be the sign of $\mu(v_0)$. Then, set $\lambda(\sigma)\coloneqq s e_k$ (where the $e_i$ are the unit vector of the standard basis of $\R^d$). By definition of $\B(K_{d+1})$, the integer $k$ is in $\{1,\ldots,d\}$. The map $\lambda$ is obviously a $\Z_2$-map, and it is simplicial $\sd(\B(K_{d+1})) \rightarrow \partial \Diamond^d$ because when $\sigma \subseteq \sigma'$ are two simplices of $\B(K_{d+1})$, the definition of $k$ and $s$ makes that $\lambda(\sigma)$ and $\lambda(\sigma')$ cannot be mapped to opposite vectors.
\end{proof}

\begin{lemma}\label{lem:tech}
There exists a (polynomially computable) piecewise affine arithmetic $\Z_2$-circuit $h\colon \R^{V(K_{d+1})}\rightarrow \R^d$ such that we have $\|z\|_{\infty} \leq d \|h(z)\|_{\infty}$ for those $z\in \R^{V(K_{d+1})}$ whose components are neither all positive, nor all negative.
\end{lemma}

\begin{proof}
    The simplicial complex $\sd(\B(K_{d+1}))$ is a subcomplex of $\sd\partial\Diamond^{d+1}$. Let $\K$ be the simplicial complex $\{0\} * \partial\Diamond^d$ (which can be seen as the triangulation of $\Diamond^d$ obtained by adding a vertex at its center). We extend $\lambda$ defined by Lemma~\ref{lem:bk} to a circuit $\lambda'$ defined on $\sd(\partial\Diamond^{d+1})$ by setting $\lambda'(v) \coloneqq 0$ on the two vertices $v$ of $\sd\partial\Diamond^{d+1}$ that are not in $\sd(\B(K_{d+1}))$ (namely, the vertices corresponding to the facets containing the all-positive and all-negative points). The circuit $\lambda'$ is a simplicial $\Z_2$-map $\sd(\partial\Diamond^{d+1}) \rightarrow \K$.

    Now, for $z \in \R^{V(K_{d+1})} \setminus\{0\}$, set $h(z) \coloneqq \|z\|_{\infty}\lambda'(z/\|z\|_{\infty})$ (where $\lambda'$ is identified with its affine extension), and set $h(0)\coloneqq 0$. This map $h$ is a piecewise affine $\Z_2$-map $\R^{V(K_{d+1})}\rightarrow \R^d$, which can be encoded as a (polynomially computable) arithmetic $\Z_2$-circuit. Since $\|y\|_{\infty} \geq \frac 1 d$ for all points $y \in \partial \Diamond^d$, the conclusion follows.
\end{proof}

\begin{proof}[Proof of Theorem~\ref{thm:complex-coind}] We establish the two reductions.

\begin{easylist}\ListProperties(Style1*=\scshape$\bullet$, Style2*=$-$, Hide=2, Indent=0.5cm, Space1=0.4cm, Space1*=0.4cm, Space2=0.2cm, Space2*=0.2cm)

# {\em Reduction of \textsc{Coind-lower-bound} to \textsc{$\varepsilon$-Borsuk--Ulam}}

## {\em Transforming the instances.}
Consider an instance of \textsc{Coind-lower-bound}. Denote by $c$ the coloring of $G$, encoded as a Boolean $\Z_2$-circuit $V(G) \rightarrow V(K_{d+1})$. Whether $c$ is proper can be checked in polynomial time. If it is not, we are done, and to formally satisfy the prescriptions of a reduction, we build any instance for \textsc{$\varepsilon$-Borsuk--Ulam}. If it is a proper coloring, then proceed as follows. For $y=\sum_{v\in V(G)}y_ve_v$, set $\bar c(y)\coloneqq  \sum_{v\in V(G)}y_ve_{c(v)}$. This gives rise to an arithmetic $\Z_2$-circuit $\bar c\colon \R^{V(G)} \rightarrow \R^{V(K_{d+1})}$. Set then $f \coloneqq h \circ \bar c \circ g$ and $\varepsilon\coloneqq \delta/d$. This way we get an instance of \textsc{$\varepsilon$-Borsuk--Ulam}.

## {\em Transforming the solutions.} If the coloring is not proper, then any monochromatic edge is returned, independently of the solution of \textsc{$\varepsilon$-Borsuk--Ulam} (type-\eqref{item:monoc} solution). Otherwise, consider a solution $x \in S^d$ of \textsc{$\varepsilon$-Borsuk--Ulam}. We have $\|f(x)-f(-x)\|_{\infty} \leq \varepsilon$, i.e., $\|f(x)\|_{\infty}\leq \frac 1 2 \varepsilon$ since $f$ is a $\Z_2$-map by construction. Write $g(x)$ as $\sum_{v\in V(G)}y_ve_v$. Set $\sigma\coloneqq \{s\cdot v\colon s\in \{+,-\},\, v \in V(G),\, s y_v >\delta/(2q)\}$, with $q=|V(G)|+1$. If $\sigma$ is not a simplex of $\B(G)$, then $x$ is a type-\eqref{item:non-embed} solution. If $\sigma$ is a simplex of $\B(G)$, then proceed as follows. For each $u \in V(K_{d+1})$, set $z_u \coloneqq \sum_{v\in c^{-1}(u)}y_v$. Because $c$ is a proper coloring, we actually have $|z_u| \geq  \sum_{v\in c^{-1}(u)\colon 2q|y_v|>\delta}|y_v| - \frac 1 2 \delta$, 
which implies $\|z\|_{\infty}+\frac 1 2 \delta \geq \|y\|_{\infty}$. Since $f(x)=h(z)$, we get $\|h(z)\|_{\infty} \leq \frac 1 2 \varepsilon$, which in turn leads to $\|z\|_{\infty} \leq\frac 1 2 \delta$, with Lemma~\ref{lem:tech} (since $\sigma$ is a simplex of $\B(G)$, the components of $z$ are neither all positive, nor all negative). We have thus $\|y\|_{\infty} \leq \delta$ and therefore $\|g(x)\|_{\infty} \leq \delta$, which means that $x$ is a type-\eqref{item:zero} solution.

# {\em Reduction of \textsc{$\varepsilon$-Borsuk--Ulam} to \textsc{Coind-lower-bound}}

## {\em Transforming the instances.} Consider an instance of \textsc{$\varepsilon$-Borsuk--Ulam}. Set $G \coloneqq K_{d+1}$ (complete graph with $d+1$ vertices), and color its vertices properly with $d+1$ colors. Set $g\colon x\in\R^{d+1} \mapsto \imath(f(-x) - f(x)) \in \R^{V(K_{d+1})}$, where $\imath\colon\R^d \xhookrightarrow{}\R^{V(K_{d+1})}$ (by identifying $\R^{V(K_{d+1})}$ with $\R^{d+1}$), and $\delta  \coloneqq \varepsilon$. This way we get an instance of \textsc{Coind-lower-bound}.

## {\em Transforming the solutions.} Consider a solution of \textsc{Coind-lower-bound}. It cannot be a type-\eqref{item:monoc} solution. It cannot be a type-\eqref{item:non-embed} solution either by property of the range of $\imath$. Thus it is a type-\eqref{item:zero} solution $x\in S^d$, such that $\|g(x)\|_{\infty}\leq \varepsilon$. Since $\imath$ is non-expansive, we have $\|f(-x)-f(x)\|_{\infty}\leq \varepsilon$, as desired.\qedhere
\end{easylist}
\end{proof}
\end{document}